\documentclass{article}

\usepackage{listings}
\usepackage{amsfonts,amssymb,amsmath,amsthm,graphics}
\usepackage{pdfpages}
\usepackage{color}
\newtheorem{theorem}{Theorem}[section]
\newtheorem{prop}[theorem]{Proposition}
\newtheorem{lemma}[theorem]{Lemma}
\newtheorem{corollary}[theorem]{Corollary}
\theoremstyle{remark}
\newtheorem{remark}{Remark}
\newtheorem{example}{Example}
\newtheorem{definition}[theorem]{Definition}

\newcommand{\inv}{ \mathrm{inv}}

\newcommand{\mbfx}{\mathbf{x}}

\usepackage{fullpage}

\begin{document}

\title{Joint $q$--moments and shift invariance for the multi--species $q$--TAZRP on the infinite line}

\author{Jeffrey Kuan}

\date{}

\maketitle

\abstract{This paper presents a novel method for computing certain particle locations in the multi--species $q$--TAZRP (totally asymmetric zero range process). The method is based on a decomposition of the process into its discrete--time embedded Markov chain, which is described more generally as a monotone process on a graded partially ordered set; and an independent family of exponential random variables. A further ingredient is explicit contour integral formulas for the transition probabilities of the $q$--TAZRP. The main result of this method is a shift invariance for the multi--species $q$--TAZRP on the infinite line.

By a previously known Markov duality result, these particle locations are the same as joint $q$--moments. One particular special case is that for step initial conditions, ordered multi--point joint $q$--moments of the $n$--species $q$--TAZRP match the $n$--point joint $q$--moments of the single--species $q$--TAZRP. Thus, we conjecture that the Airy$_2$ process describes the joint multi--point fluctuations of multi--species $q$--TAZRP. 

As a probabilistic application of this result, we find explicit contour integral formulas for the joint $q$--moments of the multi--species $q$--TAZRP in the diffusive scaling regime. }

\section{Introduction}
In the mathematics and physics community, it has been known since at least the 1990s that the KPZ equation \cite{KPZ} describes a large class of processes. For just one example, the relationship between the ASEP (totally asymmetric simple exclusion process) and KPZ was noted as early as \cite{Der97}. In \cite{bdj1}, it was shown that the asymptotic fluctuations of the longest increasing subsequence of random permutations are the same as the asymptotic fluctuations of the largest eigenvalue of the GUE (Gaussian Unitary Ensemble) \cite{TW94}. This asymptotic distribution describing the fluctuations, the Tracy--Widom $F_2$ distribution, is \textit{universal} in the KPZ universality class. More recently, it was shown that WASEP (weakly asymmetric simple exclusion process) converges to the Cole--Hopf solution of the KPZ equation \cite{ACQ}, as well as to energy solutions of the KPZ equation (see e.g. the survey \cite{PatSurvey}, which also discusses the multi--species case).

The multi--point fluctuations in the KPZ universality class are widely believed to be the Airy$_2$ process, although a rigorous proof has not yet been discovered. However, more recently, a few papers have been able to take multi-point asymptotics for partially asymmetric, non--determinantal models in the KPZ class (\cite{QS} for the ASEP and \cite{Dim} for two--point asymptotics in the stochastic six vertex model). For totally asymmetric models with a determinantal structure, such results go back to \cite{JohShape}, and many others.

In the context of \textit{multi--species} (also called  \textit{colored}, or \textit{multi--class} or \textit{mult--type}) models introduced in \cite{Ligg76}, significantly less is known, or even conjectured, about the asymptotics fluctuations. In fact, it is not even entirely clear which statistic to analyze. It has only been in recent years that non--rigorous heuristics or conjectures have appeared in the literature, such as \cite{Spohn}. Even more recently, \cite{AHR} rigorously proves that the crossing probability of a two--species model called the AHR model (introduced in \cite{AHR98}), converges to a $F_2$ times a Gaussian. Also see \cite{GMW} for another paper on multi--species ASEP. In this present paper, a different statistic is studied, perhaps indicating that multi--species models have multiple interesting statistics in the asymptotic limit. 

The model studied here is the multi--species $q$--deformed totally asymmetric zero range process (TAZRP), introduced in \cite{Take15}, generalizing the single--species version in \cite{SW98}. It was recently shown in \cite{Haya} that in the $q\rightarrow 1$ limit, the single--species $q$--TAZRP converges to the usual KPZ equation, and additionally, the fluctuations of the $q$--TAZRP are also governed by Tracy--Widom $F_2$ \cite{FVAIHP, BARRAQUAND20152674}. In words, the main result of this paper (Corollary \ref{Cor}) states that (at finite--time) the distributions of the height functions of the multi--species $q$--TAZRP can be matched to the multi--point distributions of the height functions of the single--species $q$--TAZRP. Since those asymptotic fluctuations are conjectured to be the Airy$_2$ process, it is reasonable to conjecture that the asymptotic fluctuations of the multi--species $q$--TAZRP are also the Airy$_2$ process. Actually, an even more general statement can be proved (Theorem \ref{Thm}), which is somewhat easier to prove than proving Corollary \ref{Cor} directly -- {Chronologically, the author discovered the proof of Corollary \ref{Cor} first, using the methods in section 6, before discovering the more aesthetically pleasing proof of Theorem \ref{Thm}.  } The theorem states that the height functions are invariant under certain shift transformations. 

One probabilistic application of the main result is to write explicit formulas for the joint $q$--moments in the diffusive scaling regime (i.e. when time is scaled as $L$ and space is scaled as $L^{1/2}$). These joint $q$--moments can be expressed as contour integral formulas. From these formulas, it is possible to numerically calculate the fluctuations of the multi--species $q$--TAZRP from its hydrodynamic limit.

Note that similar--looking statements to Corollary \ref{Cor} have occurred in \cite{BGW} and \cite{Gal}. The key difference here is that the initial conditions in this paper occur on the infinite line, rather than with boundary conditions. On the infinite line, asymptotic results are more tractable. As far as the author is aware, the main result in this paper do not directly follow from the results in the aforementioned papers. In any case, the method presented here is significantly different.

The outline of this paper is as follows: section 2 defines the multi--species $q$--TAZRP and states the main theorem of this paper, as well as a corollary. Section 3 briefly covers necessary background and notation. Section 4 proves the main theorem, using the methods described in the abstract. Section 5 provides the asymptotic limit in the diffusive scaling limit. Section 6 provides an alternative proof of the corollary, and section 7 contains some numerical and computer data demonstrating some examples of the results of the paper. A brief discussion on the applicability of these methods to multi--species ASEP occurs at the end of section 4.

\textbf{Acknowledgments.}
The author would like to thank Alexei Borodin and Ivan Corwin for helpful discussions. Additionally, the author would like to thank an attendee of the University of Bristol Probability Seminar, whose name the author did not catch because the seminar was held in a hybrid format, for helpful research suggestions. The numerical simulations were computed with the help of the Texas A\&M High Performance Research Computing. Finally, the author would like to thank the organizers of the workshop ``Asymptotic Algebraic Combinatorics'', held at the Institute for Pure and Applied Mathematics  in Los Angeles in February 2020, where helpful ideas were sparked.

\section{Statement of Main Result}
\subsection{Definition of the multi--species $q$--TAZRP}
Recall the definition of the $n$--species $q$--TAZRP, also known as $q$--Boson. The model was introduced in \cite{Take15}, with the single--species model going back to \cite{SW98}. The state space for $n$--species $q$--TAZRP consists of particle configurations on a one--dimensional lattice. Here, we take that lattice to be $\mathbb{Z}$. At each lattice site, there may be arbitrarily many particles, with $n$ different species of particles. The state space is therefore $\left(\mathbb{Z}_{\geq 0}^n \right)^{\mathbb{Z}}$. Each $\eta \in \left(\mathbb{Z}_{\geq 0}^n \right)^{\mathbb{Z}}$ can be written as $\eta = (\eta_i^x)$, for $x \in \mathbb{Z}$ and $1 \leq i \leq n$, where $\eta_i^x$ denotes the number of particles of species $i$ located at lattice site $x$. 

With the state space defined, the jump rates can be stated. Informally, for a particle configuration $\xi = (\xi_i^x)$, the jump rates for an $i$th species particle at lattice site $x$ to jump one step to the right is 
$$
q^{\xi_1^x + \ldots + \xi_{i-1}^x} \frac{1-q^{\xi_i^x}}{1-q},
$$
Only one particle may jump at a time. More specifically, the generator can be written explicitly. If $\eta=(\eta_i^x)$ and $\xi=(\xi_i^x)$ are related by 
$$
\eta_j^y = \xi_j^y -1, \quad \eta_j^{y+1} = \xi_j^{y+1}+1
$$
for some $j,y$, with all $\eta_i^x = \xi_i^x$ for all other values of $i,x$, then write $\eta=\xi(j,y)$. Then the generator is defined by 
$$
L_{\text{mqTAZRP}}(\xi, \xi(j,y)) =  q^{\xi_1^y + \ldots + \xi_{j-1}^y} \frac{ 1 - q^{\xi_j^y}}{1-q} .
$$
If the particle configuration $\xi(j,y)$ does not exist, then formally set 
$$
L_{\text{mqTAZRP}}(\xi, \xi(j,y)) = 0.
$$
In this notation, particles of species $i$ have priority over particles of species $j$ for $i<j$. An example of the jump rates is shown below in Figure \ref{TAZER}.

\begin{figure}
\begin{center}
\includegraphics{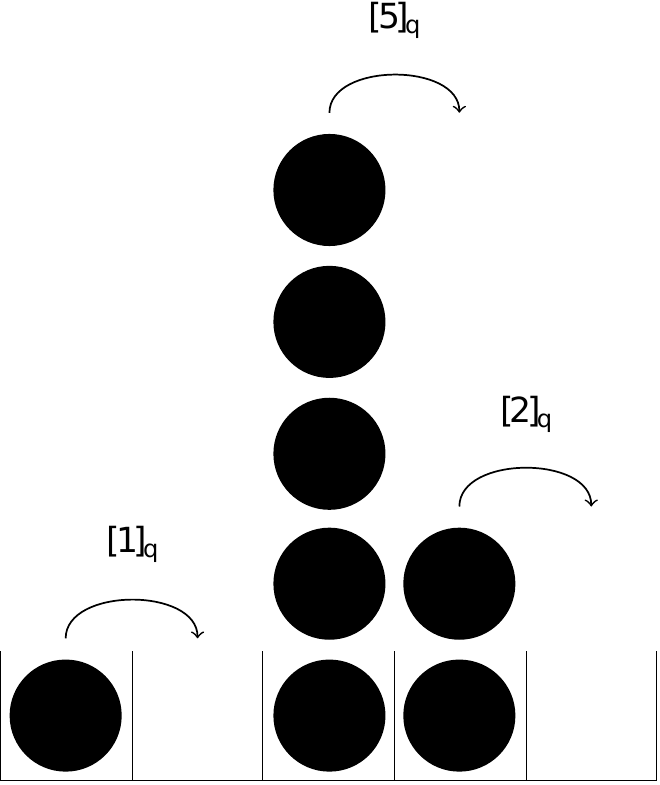}
\hspace{0.3in}
\includegraphics{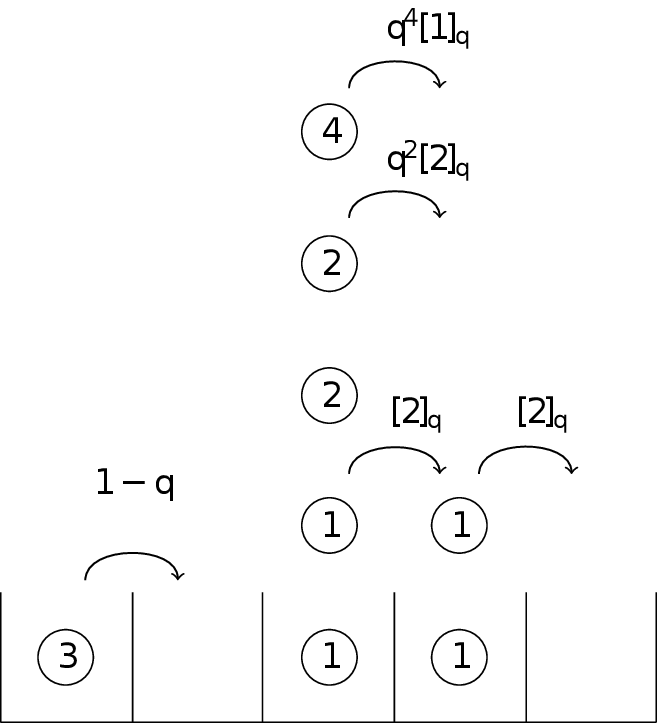}
\caption{Example of jump rates for the homogeneous $q$--TAZRP. The left image shows single--species and the right image shows multi--species. Note that $q^4[1]_q + q^2[2]_q + [2]_q  = [5]_q$.}
\label{TAZER}
\end{center}
\end{figure}

\subsection{Height Functions and Shift Invariance}
Given any $1 \leq j \leq n$ and a particle configuration $\eta=(\eta_x^i)$, set
$$
N^{(n+1-j)}_y(\eta) = \sum_{x=y}^{\infty} \sum_{i=1}^{n+1-j} \eta^x_i.
$$
In words, $N_y^{(n+1-j)}(\eta)$ counts the number of species $1,\ldots,n+1-j$ particles that are weakly to the right of $y$. The reason for using the notation $n+1-j$ in the superscript is related to the space--reversal in the duality statement of Proposition \ref{Dual} below;  These are sometimes called the \textit{height functions}.

Given a pair of particle configurations $(X,Y)$, define the \textit{intersection numbers} $I_{ij}(X,Y)$ by 
$$
I_{ij}(X,Y) = \left| [x_i,y_i] \cap [x_j,y_j] \right|
$$
for $1 \leq i,j \leq n$. These are similar to the so--called ``intersection matrices'' is made in \cite{Gal}.

With these definitions in place, the main theorem can now be stated:

\begin{theorem}\label{Thm}
Suppose that $\eta(t)$ evolves as the multi--species $q$--TAZRP and begins in initial conditions with infinitely many $i$th species particles at $y_i$ for some integers $y_1,\ldots,y_n$.  Fix lattice sites $x_1,\ldots,x_n$ (not necessarily distinct).  Let $\xi(t)$ also be a multi--species $q$--TAZRP that begins in initial conditions with infinitely many $i$th species particles at $y'_i$ for some integers $y'_1,\ldots,y_n'$. Fix lattice sites $x'_1,\ldots,x'_n$ (not necessarily distinct).  

Assume that there exists an integer $j_0 \in [1,n]$ such that 
$$
(x_1',\ldots,x_n') - (x_1,\ldots,x_n) = (y_1',\ldots,y_n') - (y_1,\ldots,y_n) = \epsilon_{j_0}
$$
where $\epsilon_{j_0}$ has a $1$ in the $j_0$ component and zeroes elsewhere. Furthermore assume that there is a sequence $\{(X(l),Y(l)):1 \leq l \leq L\}$ of pairs of particle configurations such that the intersection numbers $I_{ij}(X(l),Y(l))$ are invariant with respect to $l$ and that
$$
X(1)=(x_1',\ldots,x_n') , \quad Y(1) = (y_1',\ldots,y_n'), \quad X(L)_1 = \ldots X(L)_n
$$
and
$$
X(l+1)-X(l)=Y(l+1)-Y(l)= \epsilon_{j_l}
$$
for some $j_1, \ldots,j_K$.

Then for any non--negative integers $k_1\ldots,k_n$
$$
\mathbb{E}_y\left[ \prod_{i=1}^n q^{k_jN_{x_j}^{(n+1-j)}(\eta(t))} \right] = \mathbb{E}_{y'}\left[ \prod_{i=1}^n q^{k_jN_{x'_j}^{(n+1-j)}(\xi(t))} \right] .
$$
Furthermore, this quantity has an explicit contour integral formula. Set $N=k_1+\ldots+k_n$ and define $M_1,\ldots,M_N$ by 
$$
{M}_{k_1+\ldots+k_m+1} = {M}_{k_1+\ldots+k_m+2} = \cdots = {M}_{k_1+\ldots + k_{m+1}} = [y_m,x_m].
$$
Then
$$
\mathbb{E}\left[ \prod_{j=1}^n q^{k_j N_{x_j}^{(n+1-j)}(\eta(t))}\right] = \frac{(-1)^N q^{N(N-1)/2}}{(2\pi i)^N}\int \frac{dw_1}{w_1} \cdots \int \frac{dw_N}{w_N} B(w_1,\ldots,w_N) \prod_{j=1}^N (1-w_j)^{-{M}_j} e^{-w_jt  }   .
$$
where the $w_r$--contour contains $qw_{r+1},\ldots, qw_n$ and $1$, but not $0$.

\end{theorem}

\begin{remark}
The theorem can be applied to inductively to any sequence $x,x',x'',\ldots,$ and $y,y',y'',\ldots$ as long as the consecutive intersection numbers remain unchanged. 
\end{remark}

\begin{remark}
It is worth noting that if infinitely many $i$th species and infinitely many $j$th species particles occupy a lattice site, where $i<j$, then the $j$th species particles have jump rate $q^{\infty}=0$, and therefore never jump.
\end{remark}

By applying the theorem inductively,  the initial conditions can be shifted so that all particles start at the same lattice site:

\begin{corollary}\label{Cor}
Suppose $y_1<\ldots<y_n$ and $x_1 \geq x_2 \geq \ldots \geq x_n$. Then
$$
\mathbb{E}\left[ \prod_{j=1}^n q^{k_j N_{x_j}^{(n+1-j)}(\eta(t))}\right] = \mathbb{E}\left[ \prod_{j=1}^n q^{k_j N_{M_j}^{(n+1-j)}(\xi(t))}\right],
$$
where $\xi(t)$ is a single--species $q$--TAZRP with step initial conditions, i.e. all particles starting at $0$.
\end{corollary}

\begin{remark}
The reader may naturally be curious if this method can be applied to other models, such as the multi--species ASEP. This is discussed in section 4.4.
\end{remark}

\section{Background results and notation}

\subsection{$q$--notation}

Fix $0<q<1$. For any $k\geq 0$, let 
$$
[k]_q = \frac{1-q^k}{1-q} = 1 + q + q^2 + \ldots + q^{k-1}
$$ 
be the $q$--deformed integer. Let $[k]_q^! = [1]_q \cdots [k]_q$ be the $q$--deformed factorial. The $q$--Pochhammer symbol is
$$
(\alpha;q)_k = (1-\alpha)(1-q\alpha) \cdots ( 1-q^{k-1}\alpha), \quad 0 \leq k \leq \infty.
$$
Observe that $(1-q)^k[k]_q^! = (q;q)_k$.

For each integer $r\geq 1$ and each finite sequence of non--negative integers $\mathbf{m}=(m_1,m_2,\ldots,m_r)$ whose sum is $N=m_1 + \ldots + m_r$, define the $q$--multinomial 
$$
\left[ 
\begin{array}{c}
N\\
\mathbf{m}
\end{array}
\right]_q
:=
\left[ 
\begin{array}{c}
m_1 + \ldots + m_r \\
m_1,m_2,\ldots,m_r
\end{array}
\right]_q
= \frac{[N]^!_q}{ [m_1]^!_q \cdots [m_r]^!_q}.
$$

\subsection{Duality formula for $q$--moments}
Let $\vec{k}=(k_1,\ldots,k_n)$ and $\vec{M}=(M_1,\ldots,M_n)$. Let $ \mathcal{S}(\vec{k},\vec{M})$ denote the set of all particle configurations $\xi$ with exactly $k_j$ particles of species $j$, with all species $j$ particles located strictly to the left of $M_{n+1-j}$. More precisely,
$$
\xi \in \mathcal{S}(\vec{k},\vec{M}) \Leftrightarrow \sum_{x\in \mathbb{Z}} \xi_j^x = k_j \text{ and } \xi^z_j=0 \text{ for all } 1 \leq j \leq n , z\geq M_{n+1-j}.
$$
\begin{prop}\label{Dual}
Suppose that $\eta(t)$ begins in the initial conditions with infinitely many $i$th class particles at $-M_i$ for some integers $0 < M_n < \ldots < M_1$. Fix non--negative integers $k_1,\ldots,k_n$. Then
$$
\mathbb{E}\left[ \prod_{j=1}^n q^{k_j N_0^{(n+1-j)}(\eta)}\right] = \sum_{\xi \in \mathcal{S}(\vec{k},\vec{M})}\mathrm{Prob}_0( \xi \text{ at time } t).
$$
The initial conditions on the right--hand--side consist of $k_j$ particles of species $j$ located at $0$ for $1 \leq j \leq n$.
\end{prop}
\begin{proof}
This follows from the duality result of Theorem 2.5(b) \cite{KIMRN};  see also Corollary 5.3 of \cite{KuanCMP} for the same duality result with a different proof. The only additional observation required is that for infinitely many particles of species $i$ at a single site, the duality function only takes the values $1$ or $0=q^{\infty};$ i.e. the duality function becomes an indicator function.
\end{proof}

\subsection{Transition Probabilities for the single--species $q$--TAZRP}
Set
$$
S(w_a,w_b) = - \frac{qw_b-w_a}{qw_a-w_b}
$$
and for $\sigma \in S_n$ (the permutation group on $n$ letters),
$$
A_{\sigma} = \prod_{\substack{ \sigma(i)>\sigma(j) \\ 1 \leq i < j \leq n}} S(w_{\sigma(j)},w_{\sigma(i)}).
$$
Let 
$$
B\left(w_{1}, \ldots, w_{n}\right)=\prod_{1 \leq i<j \leq n} \frac{w_{i}-w_{j}}{w_{i}-q w_{j}}
$$
By equation (5.1) of \cite{WW},
\begin{equation}\label{WW51}
\sum_{\sigma \in S_{n}} A_{\sigma}\left(w_{\sigma^{-1}(1)}, \ldots, w_{\sigma^{-1}(n)}\right)=[n]_{q} ! B\left(w_{1}, \ldots, w_{n}\right)
\end{equation}
\begin{prop}\label{AHPThm2}
Let $\mathbf{N}=(N_1,\ldots,N_n)$ and $N=N_1+\ldots N_n$. Fix a particle configuration $x$ where $x=(x_1,\ldots,x_N)$. Given initial conditions where all $N$ particles begin at $0$,
\begin{multline*}
\mathrm{Prob}_0(X(t)=x) \text{ at time } t) =  \\
\times  \frac{(-1)^N}{W(x)} \left( \frac{1}{2\pi i}\right)^N \sum_{\omega \in S(N)} \int_{\mathcal{C}_R} \cdots \int_{\mathcal{C}_R} B(w_1,\ldots,w_N)\prod_{j=1}^N \left[ \prod_{k=y_{\omega(j)}}^{x_j} \left( \frac{1}{1 - w_{\omega(j)}}\right) e^{-w_jt}  \right] dw_1 \cdots dw_N,
\end{multline*}
where the $\mathcal{C}_R$ are large contours, centered at the origin, that enclose the poles $b_k$ and $w_j=qw_i$. Here, 
$$
W(x) = \prod_{z \in \mathbb{Z}} \frac{1}{[k_z]_q^!}, \quad k_z = \left| \{ k: x_k=z\} \right|
$$
\end{prop}
The paper \cite{KoLe} also gives the formula in terms of small contours, which will be described later. 

\begin{remark}
The results actually hold for more general initial conditions, but these will not be needed, at least until section 6.
\end{remark}

\subsection{Markov projection property}
The multi--species $q$--TAZRP has a property called the Markov projection property, which generalizes a property originally known as lumpability \cite{KS60}. This is also called the ``color--blind'' projection property. This states that if one views all particles as identical, the system is reduced to a single--species model. A specific case of this is used in the proof of the corollary. 

\section{Proof of the Main Theorem}
By the duality result, the $N$th $q$--moments can be written in terms of transition probabilities of the $N$--particle system. For the remainder of the proof, only the transition probabilities of the finite--particle system will be studied. Note that applying the duality reverses the direction of the jumps, so that the roles of $x$ and $y$ are now reversed.

\subsection{Markov processes and embedded Markov chains on posets}
The first step is to decompose the continuous--time multi--species $q$--Boson into its discrete--time embedded Markov chain and its local times. It will actually be easier to state this step in terms of a general monotone Markov process on a partially ordered set. The next subsection contains essentially a stand--alone result that can be read independently of the rest of the paper. As far as the author knows, these results are new, as are their applications in integrable probability.

The following definitions are standard in probability theory:
\begin{definition}
Consider a continuous--time Markov process $X(t)$ on a discrete state space, with generator $L(x,y)$. The \textit{discrete--time embedded Markov chain} is defined on the same state space, with transition probabilities
$$
P(x,y) = 
\begin{cases}
- \dfrac{L(x,y)}{L(x,x)}, & x \neq y\\
0, & x=y.
\end{cases}
$$
\end{definition}

\begin{definition}
Consider a  Markov process or Markov chain $X(t)$ on a state space which is also a partially ordered set with ordering denoted by $\leq$. It is \textit{monotone} if for any times $s<t$, the event $\{X(t) \geq X(s)\}$ occurs almost surely.

\end{definition}

The next proposition has possibly occurred somewhere in the literature before, but the author is unaware of any appropriate references. 

\begin{prop}\label{old}
Let $X(t)$ be a time--homogeneous monotone Markov process on a poset $\mathcal{S}$ with generator $L(x,y)$, and let $Y(n)$ denote its embedded Markov chain.  Additionally let $\{\mathcal{E}_z: z \in \mathcal{S}\}$ be independent exponential random variables with respective parameters $\{-L(z,z): z\in \mathcal{S}\}$.

Let $\Pi_{xy}$ denote the set of all sequences $\{(x=x_0 \leq z_1 < \cdots < z_n =y)\}$ with the property that every $\mathbb{P}(Y(1)=z_n \vert Y(0) = z_{n-1})$ is nonzero. In other words, there is a positive probability that $Y(k)$, starting from $x$, is located at $z_k$ (for $1\leq k \leq n$) up until it hits $y$.  For any $\pi \in \Pi_{xy}$, let 
$$\mathcal{E}_{\pi} = \sum_{z \in \pi} \mathcal{E}_z.$$ 
Then
$$
\mathbb{P}_x(X(t) \leq y) = \sum_{n \geq 0} \sum_{ \substack{ \pi \in \Pi_{xy} \\ l(\pi)=n}} \mathbb{P}_x(Y(1)=z_1,\ldots,Y(n-1)=z_{n-1},Y(n)=y) \mathbb{P}\left(  \mathcal{E}_{\pi} \geq t \right), 
$$
where $\pi=(x< z_1< \ldots < z_n = y)$ and $l(\pi)$ is the length of $\pi$.
\end{prop}
\begin{proof}
Let $\mathcal{S}^{\infty}$ denote the set of all infinite paths in $\mathcal{S}$, irrespective of the jump rates or the partial ordering on $\mathcal{S}$. As before, let $\pi=(z_0,z_1,z_2,\ldots)$, where $z_n \in \mathcal{S}$ for all $n\geq 0$. Note that the discrete embedded Markov chain $Y(t)$ and the Markov process $X(t)$ can be coupled together naturally. Thus
$$
\mathbb{P}_x(X(t) \leq y) = \sum_{\pi \in \mathcal{S}^{\infty}} \mathbb{P}(Y(n)=z_n \text{ for all } n\geq 0) \cdot \mathbb{P}_x(X(t) \leq y  \vert Y(n)=z_n \text{ for all } n\geq 0)
$$
by the usual total probability rule. By monotonicity, the first probability $ \mathbb{P}(Y(n)=z_n \text{ for all } n\geq 0) $ is only nonzero for paths that begin with $\pi \in \Pi_{xy}$ where $l(\pi) :=n \geq 0$, so the summation can be replaced with the summation in the right--hand--side of the statement of the proposition.

For any state $y$, let $\mathbb{H}(z)$ denote the holding time of $X(t)$ at $z$. Then
$$
 \mathbb{P}_x(X(t) \leq y  \vert Y(n)=z_n \text{ for all } n \leq l(\pi)) =     \mathbb{P}\left( \sum_{k=0}^{\l(\pi)} \mathbb{H}(z_k) \geq t \right)
$$
Since $\{\mathbb{H}(z): z \in \mathcal{S}\} \stackrel{d}{=} \{\mathcal{E}_z: z\in \mathcal{S} \}$, this implies that $\mathcal{E}_\pi \stackrel{d}{=}  \sum_{k=0}^{\l(\pi)} \mathbb{H}(z_k)$, this completes the proof.

\end{proof}

Now suppose that $\mathcal{S}$ is a \textit{graded poset}. In other words, there is a rank function $\rho:\mathcal{S}\rightarrow \mathbb{N}$ such that
\begin{itemize}
\item
$\rho(x) < \rho(y)$ if $x<y$.

\item
Whenever $x \lessdot y$, the rank function satisfies $\rho(x)+1=\rho(y)$.

\end{itemize}
Recall that $x\lessdot y$ means that $x<y$ and that there is no element $z$ such that $x<z<y$. In words, $x$ \textit{ is covered }by $y$. 

\begin{definition} For a discrete--time Markov chain on a discrete state space define 
$$
\mathbb{P}_{\mathcal{H},x}(y) =\mathbb{P}\left( Y(n) = y \text{ for some } n | Y(0)=x\right).
$$
In words, $\mathbb{P}_{\mathcal{H},x}(y)$ is the probability that the Markov process $X(t)$ or the chain $Y(t)$ ever hits $y$ when started from $x$. (The $\mathcal{H}$ stands for hitting). For a continuous--time Markov chain $X(t)$, similarly define
$$
\mathbb{P}_{\mathcal{H},x}(y)  =\mathbb{P}\left( X(t) = y \text{ for some } n | X(0)=x\right).
$$
Note that if $X(t)$ has discrete--time embedded Markov chain $Y(n)$, then $\mathbb{P}_{\mathcal{H},x}(y)$ is the same measure for both processes.

\end{definition}

The next definition has likely occurred before, but the author was unable to find a citation.
\begin{definition}
A discrete--time  Markov chain $Y(n)$ is \textit{compatible} with the rank function $\rho$ if the transition probability $T(x,y)$ is nonzero if and only if $x\lessdot y$. A continuous--time Markov process $X(t)$ is \text{compatible} with the rank function $\rho$ if its discrete--time embedded Markov chain is compatible with $\rho$.
\end{definition}

Before continuing, here are a few properties of compatible Markov chains and processes that will be useful later. 
\begin{lemma}\label{Used}
Suppose that the discrete--time Markov chain $Y(n)$ on a graded poset is compatible with the rank function $\rho$. Then:

(a)
Almost surely $\rho(Y(n))=\rho(Y(0))+n\text{ for all } n \geq 0$.

(b) $\mathbb{P}_{x,\mathcal{H}}(\cdot)$ is a probability measure when restricted to $\{y \in \mathcal{S}: \rho(y)=\rho(x)+n\}$ for each fixed $n$.
\end{lemma}
\begin{proof}
These both follow immediately from the definitions.
\end{proof}

With this additional assumption that $\mathcal{S}$ is a graded poset and the compatibility assumption, the next proposition can be now stated. Roughly speaking, one can prove certain distributional identities by induction on the value of the rank function. In the context of this paper, which examines the continuous--time multi--species $q$--TAZRP, part (b) below will be most useful. Part (a), which holds for both discrete and continuous--time processes, is an intermediate step to prove part (b), for the purposes of this paper. However, it is certainly possible that future applications in different models will be able to use part (a) without using part (b).

\begin{prop}\label{poset}
Suppose that $X(t)$ and $Y(n)$ are compatible with the graded poset $\mathcal{S}$ with partial ordering $\leq$.

(a) Define an equivalence relation $\equiv$ on $\mathcal{S}\times \mathcal{S}$ by 
$$
(x,y) \equiv (x',y') \text{ if } \mathbb{P}_{x,\mathcal{H}}(y) =  \mathbb{P}_{x',\mathcal{H}}(y') \text{ and } \rho(x)-\rho(y) = \rho(x')-\rho(y').
$$
Let $E^{\equiv}  $ denote the subset of $(\mathcal{S} \times \mathcal{S})^2$ which represents $\equiv$, in other words
$$
((x,y),(x',y')) \in E^{\equiv}  \text{ if and only if } (x,y) \equiv (x',y')  
$$
For any $y \in \mathcal{S}$, let $\mathcal{C}(y) \subset \mathcal{S}$ be
$$
\mathcal{C}(y)  = \{ z: z\lessdot y\}.
$$
In other words, $\mathcal{C}(y)$ is the set of all states that can possibly occur the time--step before the Markov chain hits $y$.

Then $E^{\equiv}$ satisfies the following ``inductive property'': 
\begin{itemize}
\item
If $((x,y),(x',y')) \in (\mathcal{S} \times \mathcal{S})^2$ satisfies the ``inductive step'' that there exists a bijection $\iota:\mathcal{C}(y) \rightarrow \mathcal{C}'(y)$ such that 
$$
T(z,y) = T(\iota(z),y') \text{ for all } z \in \mathcal{C}(y),\\
$$
and 
\item
if $E^{\equiv}$ satisfies as well the ``inductive hypothesis'' that 
$$
((x,z),(x',\iota(z)) \in E^{\equiv} \text{ for all } z \in \mathcal{C}(y),
$$
\end{itemize}
then $((x,y),(x',y'))$ is also in $E^{\equiv}$.

(b) Assume that all the conditions in part (a) still hold.
Define another equivalence relation $\sim$ on $\mathcal{S}\times \mathcal{S}$ by 
\begin{align*}
(x,y) \sim (x',y') \text{ if } \mathbb{P}_x(X(t) \leq y) &= \mathbb{P}_{x'}(X(t) \leq y') \text{ for all } t \geq 0\\
  \text{ and } \mathbb{P}_x(X(t) = y)   &= \mathbb{P}_{x'}(X(t) = y') \text{ for all } t \geq 0\\
 \text{ and } \rho(x)-\rho(y) &= \rho(x')-\rho(y')
\end{align*}
Let $E^{\sim}  $ denote the subset of $(\mathcal{S} \times \mathcal{S})^2$ which represents $\sim$, in other words
$$
((x,y),(x',y')) \in E^{\sim} \text{ if } (x,y) \sim (x',y') .
$$

Then $E^{\sim}$ satisfies the following ``inductive property'':
\begin{itemize}
\item
If the bijection from part (a) also satisfies
$$
L(z,z) = L(\iota(z),\iota(z)), \text{ for all } z \in \mathcal{C}(y),
$$
and furthermore, the value of $L(y,y)=L(y',y')$, and:
\item
if $E^{\sim}$ satisfies as well the ``inductive hypothesis'' that 
$$
((x,z),(x',\iota(z)) \in E^{\sim} \text{ for all } z \in \mathcal{C}(y);
$$
then $((x,y),(x',y'))$ is also in $E^{\sim}$ as long as $ \rho(x)-\rho(y) = \rho(x')-\rho(y')$.

\end{itemize}
\end{prop}
\begin{proof}
(a)
Proving that $((x,y),(x',y')) \in E^{\equiv}$ entails proving that $\rho(x)-\rho(y)=\rho(x')-\rho(y')$ and also that $ \mathbb{P}_{x,\mathcal{H}}(y) =  \mathbb{P}_{x',\mathcal{H}}(y')$. The former is easier to prove. To start, let  $n=\rho(y)-\rho(x)$.   But then $n=\rho(y')-\rho(x')$ is true as well because of Lemma \ref{Used}. Thus it remains to prove the latter statement.

Because of the compatibility criterion (Lemma \ref{Used}),
$$
 \mathbb{P}_{x,\mathcal{H}}(y) = \mathbb{P}_{x}(Y(n)=y), \quad  \mathbb{P}_{x',\mathcal{H}}(y') = \mathbb{P}_{x'}(Y(n)=y').
$$
Thus
\begin{align*}
 \mathbb{P}_{x,\mathcal{H}}(y) &= \sum_{z \in \mathcal{C}(y)} \mathbb{P}_{x}(Y(n-1)=z) \mathbb{P}_x(Y(n)=y \vert Y(n-1)=z)\\
 &=  \sum_{z \in \mathcal{C}(y)} \mathbb{P}_{x}(Y(n-1)=z) T(z,y).
\end{align*}
Likewise, 
$$
 \mathbb{P}_{x',\mathcal{H}}(y') =  \sum_{z' \in \mathcal{C}(y)} \mathbb{P}_{x'}(Y(n-1)=z') T(z',y').
$$
To show that these two quantities are equal, we will use the inductive step to match the $T$ terms and the inductive hypothesis to match the $\mathbb{P}$ terms.

Using the bijection $\iota$, the last equation can be rewritten
$$
 \mathbb{P}_{x',\mathcal{H}}(y') =  \sum_{z \in \mathcal{C}(y)} \mathbb{P}_{x'}(Y(n-1)=\iota(z)) T(\iota(z),y').
$$
By the inductive step $T(\iota(z),y') =  T(z,x')$; and by the inductive hypothesis $ \mathbb{P}_{x'}(Y(n-1)=\iota(z))=  \mathbb{P}_{x}(Y(n-1)=z)$. So thus
$$
 \mathbb{P}_{x',\mathcal{H}}(y') =  \sum_{z' \in \mathcal{C}(y)} \mathbb{P}_{x'}(Y(n-1)=z)T(z,y)  =  \mathbb{P}_{x,\mathcal{H}}(y) .
$$
This proves (a).

(b) For the continuous--time case, the argument is slightly more involved because the time steps no longer exactly line up with the rank function $\rho$. The equality $\rho(y)-\rho(x)=\rho(y')-\rho(x')$ follows again from the Lemma \ref{Used}, as in part (a).  Let $d=\rho(y)-\rho(x)=\rho(y')-\rho(x')$. 

There is the identity of events
$$
\{X(t) \leq y\} = \{X(t) = y\} \cup  \coprod_{z \in \mathcal{C}(y)} \{X(t) \leq z\},
$$
so by countable additivity
$$
\mathbb{P}_x(X(t) \leq y) = \mathbb{P}_x(X(t)=y) + \sum_{z \in \mathcal{C}(y)} \mathbb{P}_x(X(t) \leq z).
$$
Similarly,
$$
\mathbb{P}_{x'}(X(t) \leq y') = \mathbb{P}_{x'}(X(t)=y') + \sum_{z' \in \mathcal{C}(y')} \mathbb{P}_{x'}(X(t) \leq z').
$$
By the inductive hypothesis, the $\Sigma$ terms are equal, so it suffices to show
$$
\mathbb{P}_x(X(t)=y)  = \mathbb{P}_{x'}(X(t)=y') 
$$

Now, the left--hand--side can be written as (Proposition \ref{old})
$$
\mathbb{P}_x(X(t) = y) = \sum_{\pi \in \Pi_{xy}} \mathbb{P}_x(Y(1)=z_1,\ldots,Y_{d-1}=z_{d-1},Y_d=y) \mathbb{P}(\mathcal{E}_{x}+ \mathcal{E}_{z_1} + \ldots \mathcal{E}_{z_{d-1}}<t, \mathcal{E}_{\pi} \geq t),
$$
since the holding times $\mathcal{E}_{\cdot}$ are independent of the embedded Markov chain $Y(n)$. First, let us rewrite the term depending on $\mathcal{E}$'s. To save space, let $\vec{\mathcal{E}}= \mathcal{E}_x+\ldots + \mathcal{E}_{z_{d-1}}$. Since $\vec{\mathcal{E}}$ and $\mathcal{E}_{y}$ are independent, and $\mathcal{E}_{\pi} = \vec{\mathcal{E}} + \mathcal{E}_y$, this term can be rewritten as
\begin{align*}
\mathbb{P}(\vec{\mathcal{E}}<t, \vec{\mathcal{E}}+\mathcal{E}_y>t) &= \int_0^t \mathbb{P}(\vec{\mathcal{E}}<t, \vec{\mathcal{E}} + {\mathcal{E}}_y>t \vert \vec{\mathcal{E}}=s) f_{\vec{\mathcal{E}}}(s)ds \\
&= \int_0^t \mathbb{P}(\mathcal{E}_y>t-s) f_{\vec{\mathcal{E}}}(s)ds\\
&=\int_0^t e^{-\lambda_y(t-s)} f_{\vec{\mathcal{E}}}(s)ds
\end{align*}
where the notation $f_{A}(s)$ denotes the probability distribution function of a random variable $A$.
Now by integration by parts,
\begin{align*}
\mathbb{P}(\vec{\mathcal{E}}<t, \vec{\mathcal{E}}+\mathcal{E}_y>t)&=   e^{-\lambda_yt}  \left(  e^{\lambda_y s}\mathbb{P}( \vec{\mathcal{E}} \leq s) \Big|_{s=0}^t - \int_0^t   \lambda_y e^{\lambda_ys}\mathbb{P}( \vec{\mathcal{E}} \leq s)ds \right)\\
&= \mathbb{P}(\vec{\mathcal{E}} \leq t) - e^{-\lambda_y t} \int_0^t   \lambda_y e^{\lambda_ys}\mathbb{P}( \vec{\mathcal{E}} \leq s)ds.
\end{align*}
Therefore, plugging this back in to the equation above,
\begin{multline*}
\mathbb{P}_x(X(t) = y) = \sum_{\pi \in \Pi_{xy}} \mathbb{P}_x(Y(1)=z_1,\ldots,Y_{d-1}=z_{d-1},Y_d=y)\\ \times \left( \mathbb{P}(\mathcal{E}_{\pi} \leq t) - e^{-\lambda_y t} \int_0^t   \lambda_y e^{\lambda_ys}\mathbb{P}( {\mathcal{E}}_{\pi} \leq s)ds\right).\end{multline*}
Now rewriting the term that depends on $Y's$,
\begin{multline*}
\mathbb{P}_x(X(t) = y) = \sum_{z \in \mathcal{C}(y)} \sum_{\pi \in \Pi_{xz}} \mathbb{P}_x(Y(1)=z_1,\ldots,Y(d-2)=z_{d-2},Y_{d-1}=z) T(z,y)  \\
\times \left( \mathbb{P}(\mathcal{E}_{\pi} \leq t) - e^{-\lambda_y t} \int_0^t   \lambda_y e^{\lambda_ys}\mathbb{P}( {\mathcal{E}}_{\pi} \leq s)ds\right).
\end{multline*}
This then becomes
$$
\mathbb{P}_x(X(t) = y) = \sum_{z \in \mathcal{C}(y)} T(z,y)   \sum_{\pi \in \Pi_{xz}} \mathbb{P}_{\mathcal{H},x}(z) 
 \left( \mathbb{P}(\mathcal{E}_{\pi} \leq t) - e^{-\lambda_y t} \int_0^t   \lambda_y e^{\lambda_ys}\mathbb{P}( {\mathcal{E}}_{\pi} \leq s)ds\right).
$$
Now, note that
$$
\sum_{\pi \in \Pi_{xz}} \mathbb{P}_{\mathcal{H},x}(z) \mathbb{P}(\mathcal{E}_{\pi} \leq t) = \mathbb{P}_x(\tau_z + \mathcal{E}_z <t)
$$
where $\tau_z$ is the first hitting time at the state $z$, since both sides equal the probability that the process hits $z$ and then leaves $z$ before time $t$. Therefore
$$
\mathbb{P}_x(X(t) = y) = \sum_{z \in \mathcal{C}(y)} T(z,y)  
 \left( \mathbb{P}_x(\tau_z+\mathcal{E}_{z} \leq t) - e^{-\lambda_y t} \int_0^t   \lambda_y e^{\lambda_ys}\mathbb{P}_x( \tau_z + \mathcal{E}_z\leq s)ds\right).
$$
By an identical calculation,
$$
\mathbb{P}_{x'}(X(t) = y') = \sum_{z' \in \mathcal{C}(y')} T(z',y')  
 \left( \mathbb{P}_{x'}(\tau_{z'}+\mathcal{E}_{z'} \leq t) - e^{-\lambda_{y'} t} \int_0^t   \lambda_{y'} e^{\lambda_{y'}s}\mathbb{P}_{x'}( \tau_{z'} + \mathcal{E}_{z'}\leq s)ds\right).
$$
By the additional assumptions on the bijection $\iota$, we can insert $\lambda_y=\lambda_{y'}$, $ \mathcal{E}_z \stackrel{d}{=} \mathcal{E}_{\iota(z)}$ and $T(\iota(z),y')=T(z,y)$, so that the proof is completed once it is shown that
$$
\mathbb{P}_x(\tau_z<s) = \mathbb{P}_{x'}(\tau_{z'}<s) 
$$
for all $s\geq 0$. Here $z'=\iota(z)$ to save space in notation.

This final identity can be seen to follow from the inductive hypothesis. We know that$$
\mathbb{P}_x(X(t) \leq z) = \mathbb{P}_{x'}(X(t) \leq z') $$
To use this, note that for all $z \in \mathcal{C}(y)$,
\begin{align*}
\mathbb{P}_x(\tau_z<s)  &=   \lim_{m\rightarrow \infty} \sum_{a=0}^{m-1} \mathbb{P}_x\left(  \tau_z \in [as/m,(a+1)s/m) \right)\\
&=\lim_{m\rightarrow \infty} \sum_{a=0}^{m-1} \mathbb{P}_x\left(X\left(\frac{a}{m}s\right) = z, X\left(\frac{a+1}{m}s\right) \neq z\right)\\
&=  \lim_{m\rightarrow \infty} \sum_{a=0}^{m-1} \mathbb{P}_x\left(X\left(\frac{a}{m}s\right) = z\right)(1-e^{-\lambda_z/m} + O(m^{-2})),
\end{align*}
and a similar expression holds for $ \mathbb{P}_{x'}(\tau_{z'}<s) $. Since, by the induction hypothesis, $\mathbb{P}_x\left(X\left(\frac{a}{m}s\right) = z\right) = \mathbb{P}_{x'}\left(X\left(\frac{a}{m}s\right) = z'\right)$ for all $z\in \mathcal{C}(y)$, we conclude that $\mathbb{P}_x(\tau_z<s) = \mathbb{P}_{x'}(\tau_z'<s)$. This finishes the proof.
\end{proof}

\subsection{Proof of Theorem \ref{Thm}}
By Proposition \ref{Dual}, the $n$th $q$--moments of the multi--species $q$--TAZRP with infinitely many particles can be expressed in terms of the distributions of the $n$--particle system. The theorem in the previous section will be applied to these distributions, as described below.
 
The multi--species $q$--TAZRP fits into the general set-up of a stochastic Markov process on a graded poset in the last section. Let $X(t)$ denote the multi--species $q$--TAZRP and $Y(n)$ its discrete--time embedded Markov chain. Define a partial order $<$ on the state space so that $\mathbb{P}(Y(1) = y \vert Y(0)=x)$ is nonzero if and only if $x\lessdot y$. Let $\rho(\cdot)$ be the rank function that is compatible with the Markov process, and normalized so that $\rho((0,\ldots,0))=0$. In general, $\rho((x_1,\ldots,x_n))=x_1+\ldots + x_n$.

Now to proceed with the proof of the main theorem. After applying Proposition \ref{Dual}, the suffices to prove that, assuimng the intersection numbers match,
$$
\mathbb{P}_x(X(t) \leq y) = \mathbb{P}_{x'}(X(t) \leq y'),
$$
where $X(t)$ is the multi--species $q$--TAZRP and  $x=(x_1,\ldots,x_n),y=(y_1,\ldots,y_n),x'=(x_1',\ldots,x_n'),y'=(y_1',\ldots,y_n')$. A proof of this identity will be split up into two cases. The reason for requiring two cases is related to the state space of the multi--species $q$--TAZRP, in which multiple particles can occupy a site. The first case is when the $j_0$ in Theorem \ref{Thm} represents a particle which does not share a site with another particle, and the second case is when it does share a site with another particle. Theorem \ref{Thm}(b) will be applied, but actually only to the first case; the second case requires explicit contour integral formulas. 

\subsubsection{The first case: applying Proposition \ref{poset}}
Suppose that the set $\{j\in [1,n]:y_j=y_{\epsilon_{j_0}}\}$ only has one element. In words, this means that to obtain $y'$ from $y$, the $j_0$th--species particle is moved, and it does not share a lattice site with any other particle. In this case, the necessary bijection $\iota:\mathcal{C}(y) \rightarrow \mathcal{C}'(y)$ such that 
$$
T(z,y) = T(\iota(z),y') \text{ and } L(z,z) = L(\iota(z),\iota(z)) \text{ for all } z \in \mathcal{C}(y)
$$
can be constructed explicitly. More specifically, if $z=(z_1,\ldots,z_n) \in \mathcal{C}(y)$ then $y-z=\epsilon_j$ for some $1\leq j\leq n$, noting that $j$ depends on $z$. Define $\iota(z)$ by $\iota(z) = y'-\epsilon_j$ for all $z\in \mathcal{C}(y)$. By the definition of the jump rates of the multi--species $q$--TAZRP, the above identity holds, completing the proof in this case.

\begin{remark}
In the general case, computer and numerical simulations show that  $\mathbb{P}_{x,\mathcal{H}}(y)$ does \textit{not} only depend on the intersection numbers $I(x,y)$. 
\end{remark}

\subsubsection{The second case: Contour integral formulas for the transition probabilities}
At this point, the reader may have noticed that neither the Markov projection property nor the contour integral formulas have been used yet. It is at this step that both of these important ingredients are used in the proof. 

Although this section does not explicitly Proposition \ref{poset}, the proof will be set up in the general framework of a monotone Markov process on a poset and an induction argument on the rank function. In particular, the induction will be on the value of $d=\rho(y)-\rho(x)$. When $d=0$, the theorem holds because the intersection numbers are equal. Now assume that the theorem holds for all $x,y,x',y'$ such that $\rho(y)-\rho(x)=\rho(y')-\rho(x')<d$. By Proposition \ref{Dual}, this assumption implies that
$$
\mathbb{P}_x(X(t) \leq y) = \mathbb{P}_{x'}(X(t) \leq y') 
$$
for all $t\geq 0$ whenever $I(x,y)=I(x',y')$ and $\rho(y)-\rho(x)=\rho(y')-\rho(x')<d$. The next step is to manipulation the equality on the last line to obtain contour integral formulas (Proposition \ref{AHPThm2}).

We have, for general monotone Markov processes on posets, that the left--hand--side equals
\begin{align*}
\mathbb{P}_x(X(t) \leq y) &=  \mathbb{P}_x(X(t) \leq y \vert X(t)\neq x)\mathbb{P}(X(t) \neq x )  + \mathbb{P}_x(X(t) \leq y \vert X(t)=x) \mathbb{P}(X(t)=x ) \notag \\
&= \sum_{\{z: x \lessdot z\}} \mathbb{P}_x(Y(1)=z)\int_0^t \mathbb{P}_z(X(t-s) \leq y) \lambda_x e^{-\lambda_x s}ds + \int_t^{\infty} \lambda_xe^{-\lambda_x s} ds
\end{align*}
Suppose that the value of the lattice site $x_{j_0}$ occurs a total or $r$ times in $(x_1,\ldots,x_n)$. In other words, there are $r$ particles located at $x_{j_0}$. Split the sum over $\{z: x \lessdot z\}$ into two sums: the first where $z$ differs from $x$ at one of those $r$ particles, and the second is the remaining $z$. The first sum can be collapsed into a single term by the Markov projection property.  The $z$ in the first sum (call this set $S_1$) can be indexed as $z(1), \ldots, z(r)$ such that 
$$
\mathbb{P}_x(Y(1)=z(r)) = q^r \mathbb{P}_x(Y(1)=z(1)); 
$$
this follows from the definition of the jump rates. Here, we are also using the Markov projection property, which tells us that it suffices to consider the case when there is only one particle of each species. Then, we can consider $\mathbb{P}_{\text{q-exc}}(X(t-s)\leq y)$ where q--exc indicates random initial conditions supported on $z(1),\ldots,z(r)$ which are $q$--exchangeable (i.e. the probability of starting in $z(k)$ is equal to $q^{k-1}/[r]_q$. By \cite{KuanAHP}, the dynamics of multi--species $q$--TAZRP preserve $q$--exchangeability (this is proved in the proof of Lemma 3.5 of \cite{KuanAHP}). This means that the sum over $z(1), \ldots, z(r)$ can be replaced with a single term and the induction hypothesis can be applied. In other words, letting $S$ denote the
$$
S= \{x: x\lessdot z\} - \{z(2), \ldots, z(r)\},
$$
we have
\begin{align}
\mathbb{P}_x(X(t) \leq y) &=  \mathbb{P}_x(X(t) \leq y \vert X(t)\neq x)\mathbb{P}(X(t) \neq x )  + \mathbb{P}_x(X(t) \leq y \vert X(t)=x) \mathbb{P}(X(t)=x ) \notag \\
&= \sum_{z\in S} \mathbb{P}_x(Y(1)=z)\int_0^t \mathbb{P}_z(X(t-s) \leq y) \lambda_x e^{-\lambda_x s}ds + \int_t^{\infty} \lambda_xe^{-\lambda_x s} ds \label{SecondTerm}
\end{align}

For now, focus on the sum over $z$. This sum can be expressed, using the induction hypothesis, as
\begin{multline*}
\sum_{z\in S} \mathbb{P}_x(Y(1)=z)\int_0^t \mathbb{P}_z(X(t-s) \leq y) \lambda_x e^{-\lambda_x s}ds \\
=\sum_{z\in S} \mathbb{P}_x(Y(1)=z)\int_0^t  q^{N(N-1)/2}  \left( \frac{-1}{2\pi i}\right)^N \lambda_x e^{-\lambda_x s}ds\\
\times \int \cdots \int dw_1 \cdots dw_n B(w_1,\ldots,w_N) \prod_{j=1}^N \frac{ (1-w_j)^{-M_j(z)-1}}{w_j} e^{-(w_1+\cdots + w_N)(t-s)}.
\end{multline*}
Here, the integers $M_j$ depend on $z$, so we write $M_j(z)$ to emphasize this dependence, and also to point out that no other term in the integrand depends on $z$. The $M_j(z)$ also depend on $y$ but that notation will not be included.
The integral over $t$ simplifies, so that the right--hand--side above equals
\begin{multline*}
\lambda_x \sum_{z\in S} \mathbb{P}_x(Y(1)=z)  q^{N(N-1)/2}  \left( \frac{-1}{2\pi i}\right)^N \int \cdots \int dw_1 \cdots dw_n  e^{-(w_1+\ldots+w_N)t} \frac{e^{t(w_1+\ldots+w_N-\lambda_x)} - 1} {w_1+\ldots+w_N-\lambda_x}\\
\times B(w_1,\ldots,w_N) \prod_{j=1}^N \frac{ (1-w_j)^{-M_j(z)-1}}{w_j} .
\end{multline*}
This expression is analytic in $t=0$, so to compute it, it suffices to calculate its $n$th derivative $\mathcal{D}_n$ with respect to $t$ at $t=0$. The $n$th derivative will be related to the $(n-1)$th derivative, and an induction argument on $n$ will be used. The $n$th derivative of the previous equation equals 
\begin{multline}\label{DefnD}
\mathcal{D}_n(x,y): =\lambda_x \sum_{z\in S} \mathbb{P}_x(Y(1)=z)  q^{N(N-1)/2} \left( \frac{-1}{2\pi i}\right)^N  \int \cdots \int dw_1 \cdots dw_N \\
\times \frac{(-\lambda_x)^n - (-1)^n(w_1+\ldots+w_N)^n} {w_1+\ldots+w_N-\lambda_x}\frac{1}{n!}
 B(w_1,\ldots,w_N) \prod_{j=1}^N \frac{ (1-w_j)^{-M_j(z)-1}}{w_j} .
\end{multline}
To write $\mathcal{D}_{n}$ in terms of $\mathcal{D}_{n-1}$, use the identity
$$
\frac{a^n-b^n}{a-b} = a^{n-1} + a^{n-2}b + \ldots ab^{n-2}+b^{n-1} = a^{n-1} + b\frac{a^{n-1}-b^{n-1}}{a-b}
$$
with $a=\lambda_x$ and $b=w_1+\ldots+w_N$. So set $\mathcal{D}_n = \mathcal{D}_{n}^{(1)} + \mathcal{D}_{n}^{(2)}$ where
\begin{multline*}
\mathcal{D}_n^{(1)}(x,y): =-\lambda_x \sum_{z\in S} \mathbb{P}_x(Y(1)=z)  q^{N-1}q^{N-2} \cdots q^K \left( \frac{-1}{2\pi i}\right)^N  \int \cdots \int dw_1 \cdots dw_N \lambda_x^{n-1} \frac{(-1)^n}{n!}\\
\times B(w_1,\ldots,w_N) \prod_{j=1}^N \frac{ (1-w_j)^{-M_j(z)-1}}{w_j} .
\end{multline*}
Then  $\mathcal{D}_{n}^{(1)} $  simplifies as
$$
\mathcal{D}_{n}^{(1)} (x,y):  = -\lambda_x^n\frac{(-1)^n}{n!} \sum_{z\in S} \mathbb{P}_x(Y(1)=z) \mathbb{P}_x(X(0)\leq y) = -\lambda_x^n\frac{(-1)^n}{n!}
$$
But 
$$\frac{d^n }{dt^n} \Big|_{t=0} \int_t^{\infty} \lambda_xe^{-\lambda_xs}ds = \frac{d^n }{dt^n} \Big|_{t=0} e^{-\lambda_x t} = (-1)^n \frac{\lambda_x^n}{n!} $$ 
so that $\mathcal{D}_n^{(1)}$ cancels with the $n$th derivative of the second term in \eqref{SecondTerm}. In other words, 

$$
\frac{d^n }{dt^n} \Big|_{t=0} \mathbb{P}_x(X(t) \leq y)  = \mathcal{D}_n^{(2)}(x,y).
$$

Using that $\mathcal{D}_n - \mathcal{D}_{n}^{(1)} = \mathcal{D}_{n}^{(2)}$,
\begin{multline}\label{Dn2}
 \mathcal{D}_{n}^{(2)}(x,y):  = -\lambda_x \sum_{z\in S} \mathbb{P}_x(Y(1)=z)  q^{N(N-1)/2} \left( \frac{-1}{2\pi i}\right)^N  \int \cdots \int dw_1 \cdots dw_N  (w_1+\ldots+w_N)\\
 \times\frac{\lambda_x^{n-1}-(w_1+\ldots+w_N)^{n-1}}{ \lambda_x - (w_1+\ldots+w_N) }\frac{(-1)^n}{n!} B(w_1,\ldots,w_N) \prod_{j=1}^N \frac{ (1-w_j)^{-M_j(z)-1}}{w_j} .
\end{multline}
As this point, the contours themselves are relevant. The claim is that for each $ 1 \leq j \leq k \leq N$ such that $M_j(z)=M_{k}$, 
\begin{multline}\label{qRepl}
 \int \cdots \int dw_1 \cdots dw_N  w_k \frac{\lambda_x^{n-1}-(w_1+\ldots+w_N)^{n-1}}{ \lambda_x - (w_1+\ldots+w_N) }\frac{(-1)^n}{n!} B(w_1,\ldots,w_N) \prod_{j=1}^N \frac{ (1-w_j)^{-M_j(z)-1}}{w_j} \\
 =  q^{j-k}\int \cdots \int dw_1 \cdots dw_N  w_j \frac{\lambda_x^{n-1}-(w_1+\ldots+w_N)^{n-1}}{ \lambda_x - (w_1+\ldots+w_N) }\frac{(-1)^n}{n!} B(w_1,\ldots,w_N) \prod_{j=1}^N \frac{ (1-w_j)^{-M_j(z)-1}}{w_j} 
\end{multline}
To see this, it suffices to consider when $j=k-1$, as the general case follows by induction. Then the goal is to prove that
$$
 \int \cdots \int dw_1 \cdots dw_N  (w_k -q^{-1}w_{k-1}) \frac{\lambda_x^{n-1}-(w_1+\ldots+w_N)^{n-1}}{ \lambda_x - (w_1+\ldots+w_N) }\frac{(-1)^n}{n!} B(w_1,\ldots,w_N) \prod_{j=1}^N \frac{ (1-w_j)^{-M_j(z)-1}}{w_j}=0.
$$
But this is true because 
\begin{equation*}
B\left(w_{1}, \ldots, w_{N}\right)\left( \frac{q}{w_k} - \frac{1}{w_{k+1}}\right)= \left[ \prod_{\substack{1 \leq i<j \leq N  \\ (i,j) \neq (k,k+1)}} \frac{w_{i}-w_{j}}{w_{i}-q w_{j}} \right] \frac{w_{k+1}-w_k}{w_kw_{k+1}},
\end{equation*}
making the integrand anti--symmetric in $w_k$ and $w_{k+1}$, so the integral evaluates to $0$, as needed. Thus \eqref{qRepl} is true.

Therefore plugging \eqref{qRepl} into \eqref{Dn2}, we get
\begin{multline*}
 \mathcal{D}_{n}^{(2)} (x,y)=  -\lambda_x  \sum_{z\in S} \mathbb{P}_x(Y(1)=z)  q^{N(N-1)/2} \left( \frac{-1}{2\pi i}\right)^N  \int \cdots \int dw_1 \cdots dw_N  \cdot  c_{q,k}w_k \\
 \times \frac{\lambda_x^{n-1}-(w_1+\ldots+w_N)^{n-1}}{ \lambda_x - (w_1+\ldots+w_N) } \frac{(-1)^n}{n!}
 B(w_1,\ldots,w_N) \prod_{j=1}^N \frac{ (1-w_j)^{-M_j(z)-1}}{w_j} ,
\end{multline*}
where $c_{q,k}$ is some constant depending only on $q,N$ and $k$, and the $k$ can be chosen arbitrarily. (In fact, $c_{q,k}=q^{k-N}[N]_q$ but this is not needed). This can then be rewritten as
\begin{multline}\label{Comp}
 \mathcal{D}_{n}^{(2)} (x,y)=  -\lambda_x  \sum_{z\in S} \mathbb{P}_x(Y(1)=z)  q^{N(N-1)/2} \left( \frac{-1}{2\pi i}\right)^N  \int \cdots \int dw_1 \cdots dw_N  \cdot  c_{q,k} \\
 \times \frac{\lambda_x^{n-1}-(w_1+\ldots+w_N)^{n-1}}{ \lambda_x - (w_1+\ldots+w_N) } \frac{(-1)^n}{n!}
 B(w_1,\ldots,w_N) \left(\prod_{\substack{ j=1 \\ j\neq k}}^N \frac{ (1-w_j)^{-M_j(z)-1}}{w_j} \right) (1-w_k)^{-M_k(z)-1},
\end{multline}

Now suppose that $\hat{x},\hat{y},\hat{x}',\hat{y}'$ satisfy $\rho(\hat{y})-\rho(\hat{x})=d=\rho(\hat{y}')-\rho(\hat{x}')$ and $I(\hat{x},\hat{y})=I(\hat{x}',\hat{y}')$. The goal then is to prove that
$$
\mathbb{P}_{\hat{x}}(X(t) \leq \hat{y}) = \mathbb{P}_{\hat{x}'}(X(t) \leq \hat{y}') 
$$
for all $t\geq 0$. By the same arguments as before, namely that the expression is analytic in $t$, it suffices to show that the $n$th derivative with respect to $t$ at $t=0$ are equal; i.e.
$$
\mathcal{D}_n^{(2)}(\hat{x},\hat{y}) = \mathcal{D}_n^{(2)}(\hat{x}',\hat{y}') 
$$
for all $n\geq 0$. To proceed, use induction on $n$. The base case, when $n=0$ is true because both sides just equal $1$. So suppose that the expression is true for all values less than $n$. 

Returning to  $\eqref{DefnD}$, with $n-1$ in place of $n$,  plug in the series 
$$
1+(1-w_k)+(1-w_k)^2 + (1-w_k)^3 + \ldots = \frac{1}{w_k},
$$
to obtain
\begin{multline*}
\mathcal{D}_{n-1}(x,y) =\lambda_x \sum_{z\in S} \mathbb{P}_x(Y(1)=z)  q^{N(N-1)/2} \left( \frac{-1}{2\pi i}\right)^N  \int \cdots \int dw_1 \cdots dw_N \\
\times \frac{(-\lambda_x)^{n-1} - (-1)^{n-1}(w_1+\ldots+w_N)^{n-1}} {w_1+\ldots+w_N-\lambda_x}\frac{1}{(n-1)!}
 B(w_1,\ldots,w_N) \prod_{\substack{j=1 \\ j \neq k}}^N \frac{ (1-w_j)^{-M_j(z)-1}}{w_j}  \\
 \times \sum_{s=0}^{\infty} (1-w_k)^{-M_k(z)-1+s}.
\end{multline*}

At this point, the value of $k$ needs to be chosen. It will depend only on $x,y,x',y'$. Because the intersection numbers $I(x,y)$ and $I(x',y')$ are all equal, the interval $[x_{j_0},y_{j_0}]$ must be either contained in or contain every other interval $[x_j,y_j]$ for $1\leq j \leq N$. Set $k$ to be the value which minimizes the length of the interval $[x_j,y_j]$; we must have $[x_k,y_k]\subseteq [x_{j_0},y_{j_0}].$ Then $M_k=y_k-z_k$ and the sum over $s$ need only range of $0$ to $y_k-z_k$ (otherwise the integral evaluates to $0$). For each such $s$, let $y^{(s)} = y +-s \epsilon_k$. Then rearranging the summations and using this new notation, we can conclude from the induction assumption, that for all $n \geq 0$,
\begin{multline*}
\mathcal{D}_{n-1}(x,y) =\lambda_x \sum_{z\in S} \mathbb{P}_x(Y(1)=z)  q^{N(N-1)/2} \left( \frac{-1}{2\pi i}\right)^N  \int \cdots \int dw_1 \cdots dw_N \\
\times \frac{(\lambda_x)^{n-1} -(w_1+\ldots+w_N)^{n-1}} {w_1+\ldots+w_N-\lambda_x}\frac{(-1)^{n-1}}{(n-1)!}
 B(w_1,\ldots,w_N) \prod_{\substack{j=1 \\ j \neq k}}^N \frac{ (1-w_j)^{-M^{(s)}_j(z)-1}}{w_j}  \\
 \times \sum_{s=0}^{y_k-z_k} (1-w_k)^{-M^{(s)}_k(z)-1},
\end{multline*}
where we recall that the integers $M_j(z)$ depend both on $y$ and $z$, so that the superscript indicates the dependence is on $y^{(s)}$ rather than $y$. Comparing this to \eqref{Comp}, this means that
$$
\mathcal{D}_{n-1}(x,y) = \frac{n}{c_{q,k}} \sum_{s=0}^{y_k-x_k} \mathcal{D}_n^{(2)}(x,y^{(s)}),
$$
which rearranges to 
$$
 \mathcal{D}_n^{(2)}(x,y) = c_{q,k}n^{-1}\mathcal{D}_{n-1}(x,y) -\sum_{s=1}^{y_k-x_k} \mathcal{D}_n^{(2)}(x,y^{(s)}) 
$$
By the induction assumption on $n$, the first term matches, meaning that
$$
c_{q,k}n^{-1}\mathcal{D}_{n-1}(\hat{x},\hat{y}) = c_{q,k}n^{-1}\mathcal{D}_{n-1}(\hat{x}',\hat{y}') .
$$
By the induction assumption on $d$, and the choice of $k$ (so that the intersection numbers match), the sum matches too. This completes the proof.

\subsection{Remarks on multi--species ASEP}
A natural question is if this method applies to the multi--species ASEP. Based on numerical simulations, the shift invariance of the finite--particle multi--species ASEP on the line will still hold. A proof of this statement could possibly follow from the general set up of Markov processes and embedded Markov chains on graded posets, although in the ASEP case the process is no longer monotone. The author hopes to pursue this step in later work.

\section{Asymptotic Limits}
As an application of Theorem \ref{Thm}, we consider the asymptotic limits of the joint $q$--moments of the multi--species $q$--TAZRP. 

\begin{theorem}\label{Application}
Use the same notations as in Theorem \ref{Thm}. Let $t,M_1,\ldots,M_N$ depend on a parameter $L$ in such a way that 
$$
t/L\rightarrow 1, \quad L^{-1/2}(M_j - L)\rightarrow \sigma_j.
$$
Then as $L\rightarrow \infty$,
\begin{multline*}
\mathbb{E}_y\left[ \prod_{j=1}^n q^{k_j N_{x_j}^{(n+1-j)}(\eta(t))}\right] \\
\rightarrow  \int_{-\infty}^{\sigma_1} dv_1 \cdots \int_{-\infty}^{\sigma_N} dv_N \frac{q^{N(N-1)/2}}{(2\pi L^{1/2})^N} \int_{-\infty}^{\infty} du_1 \cdots  \int_{-\infty}^{\infty} du_N  B(u_1,\ldots,u_N) \prod_{j=1}^N \exp(-i v_j u_j -u_j^2/2).
\end{multline*}
\end{theorem}
\begin{proof}
Recall that the contour integral formula is:
$$
\frac{(-1)^N q^{N(N-1)/2}}{(2\pi i)^N}\int \frac{dw_1}{w_1} \cdots \int \frac{dw_N}{w_N} B(w_1,\ldots,w_N) \prod_{j=1}^N (1-w_j)^{-{M}_j} e^{-w_jt  } 
$$
Let $t$ and $M_j$ depend on a parameter $L$ as $t=L$ and $M_j = \lfloor  L + \sigma_j L^{1/2}\rfloor$. By the usual steepest descent method, the integrand is asymptotically
$$
\frac{B(w_1,\ldots,w_N)}{w_1\cdots w_N} \exp(S(w_1)) \cdots \exp(S(w_N)) $$
where $S(w)$ is the function 
$$
S(w) = -w - \log(1-w).
$$
The contours need to be deformed to regions where the real part of $S(w)$ is negative; this is shown in Figure \ref{OnlyFig}. Explicitly, the contours can be circles of radius $1$ centered at $1$. However, then the contours pass through the simple poles at $0$. 
\begin{figure}
\begin{center}
\includegraphics[height=8cm]{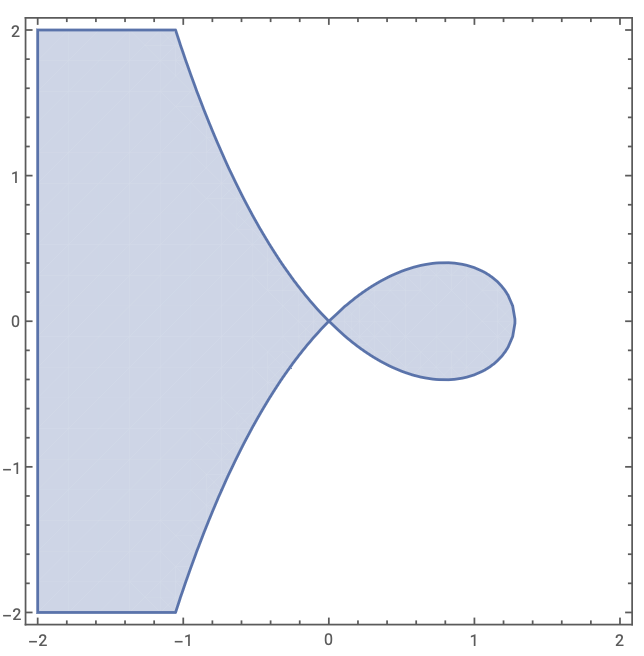}
\end{center}
\caption{The shaded region is where $\Re S(w)>0$. The contours must go through the unshaded region.}
\label{OnlyFig}
\end{figure}

To circumvent the issue, we will take the partial derivatives with respect to $\sigma_j$. By elementary analysis arguments, the partial derivatives can be interchanged with the integral and with the limit. Then the contour integral becomes, in the asymptotic limit,
$$
\frac{ q^{N(N-1)/2}}{(2\pi i)^N}\int \frac{\log(1-w_1) }{w_1}dw_1 \cdots \int \frac{\log(1-w_N)}{w_N} dw_NB(w_1,\ldots,w_N) \prod_{j=1}^N\exp(LS(w_j) -L^{1/2}\sigma_j\log(1-w_j))
$$
which no longer has poles at $0$. Near $w=0$, the function $S(w)$ and $\log(1-w)$  has the Taylor series
$$
S(w) = \frac{w^2}{2} + \frac{w^3}{3} + O(w^4), \quad \log(1-w) = -w- \frac{w^2}{2} +O(w^3).
$$
Therefore we make the substitution $w_j = -i L^{-1/2}u_j$. The $u_j$--contours then becomes the real line, so the limit becomes

$$
 \frac{q^{N(N-1)/2}}{(2\pi L^{1/2})^N} \int_{-\infty}^{\infty} du_1 \cdots  \int_{-\infty}^{\infty} du_N  B(u_1,\ldots,u_N) \prod_{j=1}^N \exp(-i \sigma_j u_j -u_j^2/2).
$$
Here we used that $B(cu_1,\ldots,cu_N)=B(u_1,\ldots,u_N)$ for any constant $c$. Finally, we need to integrate over $\sigma_1,\ldots,\sigma_N$ since the partial derivatives with respect to $\sigma_j$ were taken before. 

\end{proof}
Note that the integral is real, despite the presence of the imaginary number in the integrand. In fact, for $N=1$ recall that
$$
\frac{1}{\sqrt{2\pi}} \int_{-\infty}^{\infty} e^{-i\sigma u-u^2/2}du = e^{-\sigma^2/2}.
$$
Also note that the $N=1$ case does not depend on $q$. This can be understood in the following way: as $q$ decreases towards $0$, the height function $N_x(\eta(t))$ also decreases, because the jump rates slow down. Both of these decreases cancel out in the value of the expectation $\mathbb{E}[q^{N_x(\eta(t))}]$. 

In principle, this theorem allows us to find (computationally) the joint distributions of the height functions $N_{x}^j$. The author hopes to provide an explicit, non--numerical formula in later work.

\section{Alternative Proof Of the Corollary \ref{Cor}}
This section provides an alternative proof of the corollary. The argument uses the Markov duality and the transition probabilities for the multi--species $q$--TAZRP with $q$--exchangeable initial conditions. Because the corollary has already been proved with a more general method, and because this alternative proof is specific to the $q$--TAZRP and not to other models like ASEP, and because this proof uses some more algebraic notation, some of the exposition here is deliberately kept terse.

\subsection{More notation}

In general, a particle configuration can have infinitely many particles. When restricting to states with finitely many particles, there is another convenient way of writing particle configurations.  Assume there are $N_k$ particles of species $k$ ($1 \leq k \leq n$). Set $N=N_1 + \ldots + N_n$ to be the total number of particles. Let $\mathbf{N}$ denote $(N_1,\ldots,N_n)$. For $i\leq j$ let $N_{[i,j]}$ denote $N_i + \ldots + N_j$. A particle configuration can be expressed as a pair $(\mathbf{x},\sigma)$, where 
$$
\mathbf{x} = (x_1 \geq x_2 \geq \ldots \geq x_N)
$$
indicates the location of the particles. Let $\sigma \in S(N)$ denote the ordering of the species, in the sense that if $\sigma$ is written in two--line notation as
$$
\left(
\begin{array}{cccc}
\sigma_1 & \sigma_2 & \cdots & \sigma_N\\
1 & 2 & \cdots & N
\end{array}
\right),
$$
so that $\sigma_j = \sigma^{-1}(j)$, then the $N_k$ particles of species $k$ are located at the lattice sites
\begin{equation}\label{Equiv1}
x_{\sigma_{N_{[1,k-1]}+1}}, \ldots, x_{ \sigma_{  N_{[1,k]} }}.
\end{equation}

An equivalent description of the particle configuration $(\mathbf{x},\sigma)$ is as follows. For
$$
\mathbf{x} \in \mathcal{W}_N := \{ (x_1 , \ldots, x_N) : x_1 \geq \ldots \geq x_N  \} \subset \mathbb{Z}^N,
$$
define $\mathbf{m}(\mbfx)=(m_1,\ldots,m_r)$ so that 
$$
x_1 = \cdots = x_{m_1} > x_{m_1+1} = \cdots = x_{m_1+m_2} > x_{m_1+m_2+1} = \cdots = \cdots
 > x_{m_1+\ldots + m_{r-1}+1} = \cdots = x_N,
$$
 where $m_r$ is defined by $m_1+\ldots+m_r=N$.  Also define $k_1,\ldots,k_N$ by $k_1 = \cdots = k_{N_1}=1, k_{N_1+1} = \cdots = k_{N_1+N_2}=2, \ldots$. Then the particles located at the lattice site $x_{m_1 + \ldots + m_s + 1} = \cdots = x_{m_1 + \ldots + m_{s+1}}$ have species 
\begin{equation}\label{Equiv2}
k_{\sigma(m_1 + \ldots + m_s + 1)}, \ldots, k_{\sigma(m_1 + \ldots + m_{s+1})}.
\end{equation}
Note that $\sigma(j)$ is not the same as $\sigma_j$ in this notation.

Because more than one particle can occupy a site, the map $\mathcal{W}_N \times S(N) \rightarrow  (\mathbb{Z}^n_{\geq 0})^{\mathbb{Z}}$ is not injective. 

\textbf{Example 1.} Consider the particle configuration shown in the left side of Figure \ref{State}. There is more than one $\sigma\in S(N)$ which defines this particle configuration, and it is not hard to see that $\sigma = 21467358$ has the fewest inversions. The permutations $21476358,21567438,$ and $ 35178426$ also describe the same particle configuration as $\sigma$, but have more inversions. 

\begin{figure}
\begin{center}
\includegraphics{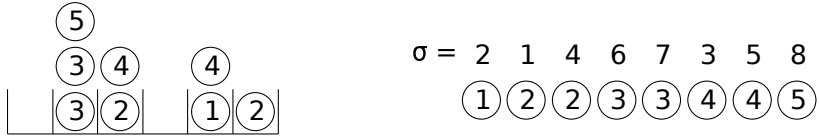}
\caption{The particle configuration referenced in Example 1.} 
\label{State}
\end{center}
\end{figure}

From now on, a particle configuration of the $n$--species $q$--TAZRP will be uniquely defined by a pair $(\mathbf{x},\sigma)$, where $\mathbf{x} \in \mathcal{W}_N$ and $\sigma \in S(N)$ has the fewest inversions among the permutations that define the particle configuration. Uniqueness is proven in Proposition 2.7 of \cite{KuanAHP}.

Let $\inv(\sigma)$ denote the number of inversions of $\sigma$.

Given a particle configuration $(\mathbf{x},\sigma)$, where $\mathbf{m}(\mathbf{x})=(m_1,\ldots,m_r)$ and the species numbers are given by $(N_1,\ldots,N_n)$, let $L_{ij}$ denote the number of species $j$ particles located at the lattice site $x_{m_1 + \ldots + m_i}$. Here the ranges of $i$ and $j$ are given by $1 \leq i \leq r$ and $1 \leq j \leq n$.

Let us define the dynamics of the multi--species $q$--TAZRP. A visual example is shown below in Figure \ref{TAZER}.

Consider a particle configuration $\xi$ of the $N$--particle $q$--TAZRP with exactly one particles of each species. For any $\mathbf{N}=(N_1,\ldots,N_n)$ such that $N_1+\ldots+N_n=N$, there is a projection map $\Pi_{\mathbf{N}}$ defined by 
$$
(\Pi_{\mathbf{N}}(\xi))^x_i = \xi^x_{N_1+\ldots+N_i+1} + \ldots +  \xi^x_{N_1+\ldots+N_{i+1}} 
$$
In words, $\Pi_\mathbf{N}$ forgets the distinction between the species $N_1+\ldots+N_{i}+1,\ldots, N_1+\ldots+N_{i+1}$, where $1\leq i \leq n$. It is not hard to see from the jump rates that
$$
\mathrm{Prob}_{\Pi_{\mathbf{N}}\eta}( \Pi_{\mathbf{N}}\xi;t) = \mathrm{Prob}_{\eta}(\xi;t).
$$
In other words, $\Pi_{\mathbf{N}}$ is a Markov projection.

We also remark on the connection to $q$--TASEP and the Tracy--Widom fluctuations. Let $x_t(t)>x_2(t)>\ldots$ denote the positions of the $q$--TASEP. Then $\eta^k(t) = x_{k}(t) -x_{k+1}(t)-1$, where $x_0(t)$ is set to be $\infty$.  Thus
$$
N_k(\eta(t)) = \eta^k(t) + \eta^{k+1}(t) + \ldots = \lim_{L\rightarrow \infty} x_k(t) - x_{k+L}(t) - L = x_k(t) + k
$$
for step initial condition $x_k(0)=-k$. 

\subsection{Transition Probabilities in the multi--species case}
The following describes some transition probabilities in the multi--species case. 
\begin{equation}\label{WW51}
\sum_{\sigma \in S_{n}} A_{\sigma}\left(w_{\sigma^{-1}(1)}, \ldots, w_{\sigma^{-1}(n)}\right)=[n]_{q} ! B\left(w_{1}, \ldots, w_{n}\right)
\end{equation}
\begin{prop}\label{AHPThm}
Let $\mathbf{N}=(N_1,\ldots,N_n)$ and $N=N_1+\ldots N_n$. Fix a particle configuration $(\mathbf{x},\sigma)$ where $\mathbf{x}=(x_1,\ldots,x_N)$ and $\sigma \in S(N)$. Given initial conditions where all $N$ particles begin at $0$,
\begin{multline*}
\mathrm{Prob}_0((\mathbf{x},\sigma) \text{ at time } t) = q^{\inv(\sigma)}\left(\frac{\prod_{j=1}^{n}\left[N_{j}\right]_{q}^{!}}{\prod_{j=1}^{n} \prod_{i=1}^{r}\left[L_{i j}\right]_{q}^{!}}\right)   \left( \prod_{k=1}^N \frac{-1}{b_{x_k}}\right) \\
\times  \left( \frac{1}{2\pi i}\right)^N \sum_{\omega \in S(N)} \int_{\mathcal{C}_R} \cdots \int_{\mathcal{C}_R} B(w_1,\ldots,w_N)\prod_{j=1}^N \left[ \prod_{k=y_{\omega(j)}}^{x_j} \left( \frac{b_k}{b_k - w_{\omega(j)}}\right) e^{-w_jt}  \right] dw_1 \cdots dw_N,
\end{multline*}
where the $\mathcal{C}_R$ are large contours, centered at the origin, that enclose the poles $b_k$ and $w_j=qw_i$.
\end{prop}
\begin{proof}
This is a special case of Theorem 3.2 of \cite{KuanAHP}, and follows from Corollary 1.3 of \cite{Wawa}; see also \cite{KoLe} for the homogeneous case when all $b_x\equiv 1$.
\end{proof}

\subsection{Completing the Alternative Proof}
Set
$$
c(q,N,K,m) = \sum_{K \leq i_1<i_2<\cdots <i_m\leq N-1} q^{i_1}q^{i_2}\cdots q^{i_m}
$$
For example,
$$
c(q,N,K,0) =1, \quad \quad c(q,N,K,1) = q^{N-1} + q^{N-2} + \ldots + q^{K}, \quad \quad c(q,N,K,N-K) = q^{N-1}q^{N-2}\cdots q^K.
$$
Note that (splitting the sum into $i_1=K-1$ and $i_1\neq K-1$)
\begin{equation}\label{cCur}
c(q,N,K-1,a) 
= 
c(q,N,K,a) + q^{K-1}c(q,N,K,a-1), 
\end{equation}
For $\sigma \in S(K)$ and $K\leq L \leq N$, define the ``large contour'' integral
\begin{multline*}
\mathcal{I}(q,N,L,K,\vec{M},\sigma,\vec{x}) = c(q,N,K,L-K)q^{\inv(\sigma)} (-1)^L \left(\frac{1}{2\pi i}\right)^L \int_{\mathcal{C}_R} dw_1 \cdots  \int_{\mathcal{C}_R} dw_L B(w_1,\ldots,w_L) \\
\times \prod_{j=1}^{L-K} \left( \frac{(1-w_j)^{-M_j-1}}{w_j}\right) \prod_{j=L-K+1}^L  \left((1-w_j)^{-x_{\sigma(j+K-L)}-1} \right) e^{-(w_1+\cdots w_L)t},
\end{multline*}
where by convention an empty product $\prod_{j=a+1}^a$ equals $1$.
Additionally define the ``small contour'' integral by
\begin{multline*}
\tilde{\mathcal{I}}(q,N,K,\vec{M},\sigma,\vec{x}) = q^{N-1}q^{N-2}\cdots q^{K} q^{\inv(\sigma)}(-1)^N \left(\frac{1}{2\pi i}\right)^N \int_{\tilde{\mathcal{C}}^1} dw_1 \cdots  \int_{\tilde{\mathcal{C}}^N} dw_N B(w_1,\ldots,w_N) \\
\times \prod_{j=1}^{N-K} \left( \frac{(1-w_j)^{-M_j-1}}{w_j}\right) \prod_{j=N-K+1}^N  \left((1-w_j)^{-x_{\sigma(j+K-N)}-1} \right) e^{-(w_1+\cdots w_N)t},
\end{multline*}
where the contour $\tilde{\mathcal{C}}^r$ contains $q\tilde{\mathcal{C}}^{r+1}, \ldots, q\tilde{\mathcal{C}}^N$ and $1$, but not $0$. 

The next proposition gives both a ``large contour'' and a ``small contour'' integral formula for the probability mentioned in the abstract. Note that the two formulas reveal the structure of the two Markov projections. 

\begin{prop}\label{Main}
Fix all $b_x \equiv 1$. Consider a $N$--particle $q$--TAZRP with exactly one particle of each species. Let the initial condition consist of all particles at $0$.   For $1\leq i \leq N$, let $z_i(t)$ denote the location of the $i$th--species particle at time $t$. For any fixed $K <  N$ and any $x_1,\ldots,x_K \geq 0$, let $\sigma\in S(K)$ denote the permutation satisfying $x_{\sigma(1)}\geq \cdots \geq x_{\sigma(K)}$ with the fewest number of inversions. Then

(a)
$$
\mathrm{Prob}(z_1(t)=x_1,\ldots,z_K(t)=x_K, z_{K+1}(t) \leq M_{N-K}, \ldots, z_N(t) \leq M_1 )=\sum_{L=K}^N \mathcal{I}(q,N,L,K,\vec{M},\sigma,\vec{x}).
$$

(b)
$$
\mathrm{Prob}(z_1(t)=x_1,\ldots,z_K(t)=x_K, z_{K+1}(t) \leq M_{N-K}, \ldots, z_N(t) \leq M_1 )= \tilde{\mathcal{I}}(q,N,K,\vec{M},\sigma,\vec{x}).
$$

\end{prop}
\begin{proof} (a)
Proceed by induction on the value of $N-K$. The base case is when $K=N-1$. We then need to show that the probability on the left--hand--side equals (where $\iota_L = (-2\pi i)^L$)
\begin{align*}
&q^{\inv(\sigma)} \iota_{N-1} \int_{\mathcal{C}_R} dw_1\cdots \int_{\mathcal{C}_R} dw_{N-1} B(w_1,\ldots,w_{N-1})\prod_{j=1}^{N-1} \left( (1- w_j)^{-x_{\sigma(j)}-1} e^{-w_jt}\right)\\
&+ q^{N-1}q^{\inv(\sigma)} \iota_N\int_{\mathcal{C}_R}  dw_1\cdots \int_{\mathcal{C}_R}  dw_NB(w_1,\ldots,w_N) \prod_{j=2}^N \left( (1- w_j)^{-x_{\sigma(j-1)}-1} e^{-w_jt}\right) \left( \frac{(1- w_1)^{-M_1-1} }{w_1}e^{-w_1t}\right).
\end{align*}
Define $\bar{x}=(\bar{x}_1 \geq \ldots \geq \bar{x}_{N-1})$ so that $x_{\sigma(j)}=\bar{x}_{j}$. By Proposition \ref{AHPThm}, and the additivity property of inversions (namely, that $\mathrm{\inv}(ab)=\inv(a)+\inv(b)$),
\normalsize
\begin{align*}
&\mathrm{Prob}(z_1(t)=x_1,\ldots,z_K(t)=x_{N-1}, z_N(t) \leq M_1 )\\
&= q^{\inv(\sigma)} \sum_{k=0}^{N-1} q^{N-1-k} \sum_{y = \bar{x}_{k+1}+1}^{\bar{x}_{k}} (-1)^N \left( \frac{1}{2\pi i}\right)^N  \int_{\mathcal{C}_R} \cdots \int_{\mathcal{C}_R} B(w_1,\ldots,w_N) \\
& \quad \quad \times \left[ \prod_{j =1}^{k} (1-w_j)^{-\bar{x}_j-1} \right] (1-w_{k+1})^{-y-1} \left[ \prod_{j=k+2}^N (1-w_j)^{-\tilde{x}_{j-1}-1}\right] e^{-(w_1+\ldots+w_N)t} dw_1,\ldots dw_N,
\end{align*}
\normalsize
where by convention $\bar{x}_0=M_1$ and $\bar{x}_N+1=0$, and a sum of the form $\sum_{a+1}^a$ equals $0$, and a product of the form $\prod_{a+1}^a$ equals $1$. Each summation is a geometric series, so 
\begin{align*}
q^{N-1-k}\sum_{y = \bar{x}_{k+1}+1}^{\bar{x}_{k}} (1-w_{k+1})^{-y-1} &= q^{N-1-k} \frac{1- (1-w_{k+1})^{\bar{x}_{k+1}-\bar{x}_{k}}}{ 1- (1-w_{k+1})^{-1}} (1-w_{k+1})^{-\bar{x}_{k+1}-2}\\
&= q^{N-1-k} \frac{(1-w_{k+1})^{-\bar{x}_{k+1}-1}- (1-w_{k+1})^{-\bar{x}_{k}-1}}{-w_{k+1}} .
\end{align*}
Note that the identity still holds when $\bar{x}_j=\bar{x}_{j+1}$.
Therefore
\normalsize
\begin{align*}
&\mathrm{Prob}(z_1(t)=x_1,\ldots,z_K(t)=x_{N-1}, z_N(t) \leq M_1)\\
&=q^{\inv(\sigma)}  \sum_{k=0}^{N-1}  (-1)^{N-1} \left( \frac{1}{2\pi i}\right)^N  \int_{\mathcal{C}_R} \cdots \int_{\mathcal{C}_R} B(w_1,\ldots,w_N) dw_1\cdots dw_N\\
& \quad \quad \times  \left[ \prod_{j =1}^{k} (1-w_j)^{-\bar{x}_j-1} \right]   q^{N-1-k} \frac{(1-w_{k+1})^{-\bar{x}_{k+1}-1}- (1-w_{k+1})^{-\bar{x}_{k}-1}}{w_{k+1}}  \left[ \prod_{j=k+2}^N (1-w_j)^{-\tilde{x}_{j-1}-1}\right] e^{-(w_1+\ldots+w_N)t} .
\end{align*}
\normalsize
This sum turns out to be a telescoping sum. To see this, we need to show that for each $1 \leq k \leq N-1$,
\begin{multline*}
\int_{\mathcal{C}_R} \cdots \int_{\mathcal{C}_R} B(w_1,\ldots,w_N) \left[ \prod_{j =1}^{k-1} (1-w_j)^{-\bar{x}_j-1} \right] (1-w_k)^{-\bar{x}_{k}-1} (1-w_{k+1})^{-\bar{x}_{k}-1}  \left[ \prod_{j=k+2}^N (1-w_j)^{-\tilde{x}_{j-1}-1}\right] \\
\times \left( \frac{q^{-(k+1)}}{w_{k+1}} - \frac{q^{-k}}{w_{k}}  \right)e^{-(w_1+\ldots+w_N)t} dw_1 \cdots dw_N = 0.
\end{multline*}
This integral equals $0$ because
\begin{equation}\label{AntiSymm}
B\left(w_{1}, \ldots, w_{N}\right)\left( \frac{q}{w_k} - \frac{1}{w_{k+1}}\right)= \left[ \prod_{\substack{1 \leq i<j \leq N  \\ (i,j) \neq (k,k+1)}} \frac{w_{i}-w_{j}}{w_{i}-q w_{j}} \right] \frac{w_{k+1}-w_k}{w_kw_{k+1}},
\end{equation}
making the integrand anti--symmetric in $w_k$ and $w_{k+1}$. Thus its integral is zero, showing that the sum over $k$ is a telescoping sum. After all cancelations, the remaining terms equal
\begin{multline*}
 q^{\inv(\sigma)} (-1)^{N-1} \left( \frac{1}{2\pi i}\right)^N  \int_{\mathcal{C}_R} dw_1 \cdots \int_{\mathcal{C}_R} dw_NB(w_1,\ldots,w_N) \left[\prod_{j=1}^{N-1} (1-w_j)^{-\bar{x}_j-1}\right] \frac{(1-w_N)^{-\bar{x}_N-1}}{w_N}e^{-(w_1+\ldots+w_N)t} \\
 + q^{N-1}q^{\inv(\sigma)} (-1)^N \left( \frac{1}{2\pi i}\right)^N  \int_{\mathcal{C}_R} dw_1 \cdots \int_{\mathcal{C}_R} dw_NB(w_1,\ldots,w_N) \frac{(1-w_1)^{-\bar{x}_0-1}}{w_1} \left[ \prod_{j=2}^N (1-w_j)^{-\tilde{x}_{j-1}-1}\right]e^{-(w_1+\ldots+w_N)t} .
\end{multline*}
Recalling that $\bar{x}_N+1=0$, the $w_N$--contour of first term has only a simple pole at $w_N=0$. Using that $B(w_1,\ldots,w_{N-1},0)=B(w_1,\ldots,w_{N-1})$, the first term reduces to the first term in the statement of the lemma. It is immediate that the second terms match as well. This completes the base case. 

The inductive step follows from \eqref{cCur} and the same argument.

(b) Start with the sum in the right--hand--side of part (a). The idea is that deforming the large contours to small contours results in many cancelations, leaving only the $N$--fold contour integral. Define a ``mixed contour'' integral by
\begin{multline*}
J(q,K,L,P,\vec{M},\sigma,\vec{x}) = q^{\inv(\sigma)} (-1)^L \left(\frac{1}{2\pi i}\right)^L \int_{\mathcal{C}_R} dw_1 \cdots  \int_{\mathcal{C}_R} dw_P \int_{\tilde{C}^{P+1}} dw_{P+1} \cdots \int_{\tilde{\mathcal{C}}^L} dw_L \\
\times B(w_1,\ldots,w_L) \prod_{j=1}^{L-K} \left( \frac{(1-w_j)^{-M_j-1}}{w_j}\right) \prod_{j=L-K+1}^L  \left((1-w_j)^{-x_{\sigma(j+K-L)}-1} \right) e^{-(w_1+\cdots w_L)t}.
\end{multline*}
Note that the deformation from the large contour to the small contour only passes through the residue at $0$ (if it exists), and none of the residues of the form $qw_i=w_j$ in $B(w_1,\ldots,w_L)$. Therefore
\begin{equation}\label{IJ}
J(q,K,L,L-K,\vec{M},\sigma,\vec{x}) c(q,N,K,L-K) = \mathcal{I}(q,N,L,K,\vec{M},\sigma,\vec{x})
\end{equation}
and
\begin{equation}\label{Recur}
J(q,K,L,P-1,\vec{M},\sigma,\vec{x}) = J(q,K,L,P,\vec{M},\sigma, \vec{x}) + q^{P-L} J(q,K,L-1,P-1,\vec{M},\sigma,\vec{x}),
\end{equation}
with the latter identity using that 
$$
B(w_1,\ldots,w_{P-1},0,w_{P+1},\ldots,w_L) = (q^{-1})^{L-P}B(w_1,\ldots,w_{P-1},w_{P+1},\ldots,w_L) 
$$

Now, note that the expression in (a) equals (by \eqref{IJ}) 
$$
\sum_{L=K}^N J(q,K,L,L-K,\vec{M},\sigma,\vec{x}) c(q,N,K,L-K)
$$
while the expression in (b) equals
$$
 J(q,K,N,0,\vec{M},\sigma,\vec{x})c(q,N,K,N-K) = \sum_{L=N}^N J(q,K,L,L-N,\vec{M},\sigma,\vec{x}) c(q,N,K,L-K).
$$
Comparing these two expressions, it is seen that it suffices to show that there exist constants $\tilde{c}(q,N,K,L-K,a)$ such that
\begin{multline}\label{Dis}
\sum_{L=K+a}^N J(q,K,L,L-K-a,\vec{M},\sigma,\vec{x}) \tilde{c}(q,N,K,L-K,a) \\
= \sum_{L=K+a+1}^N J(q,K,L,L-K-a-1,\vec{M},\sigma,\vec{x}) \tilde{c}(q,N,K,L-K,a+1)
\end{multline}
where the constants $\tilde{c}(q,N,K,m,a)$ satisfy
$$
\tilde{c}(q,N,K,L-K,0) = c(q,N,K,L-K), \quad \tilde{c}(q,N,K,N-K,N-K) = c(q,N,K,N-K).
$$
We will show that \eqref{Dis} holds
$$
\tilde{c}(q,N,K,m,a) = q^Kq^{K+1}\cdots q^{K+a-1}c(q,N,K+a,m-a).
$$

Start by writing (dropping the $\vec{M},\sigma,\vec{x}$ dependence from the notation)
\begin{align*}
&\sum_{L=K+a}^N J(q,K,L,L-K-a) \tilde{c}(q,N,K,L-K,a) \\
 \stackrel{\eqref{cCur}}{=} &\sum_{L=K+a}^N J(q,K,L,L-K-a) q^Kq^{K+1}\cdots q^{K+a-1} \\
& \quad \quad \quad \times  (c(q,N,K+a+1,L-K-a) + q^{K+a}c(q,N,K+a+1,L-K-a-1))\\
=&\sum_{L=K+a}^N J(q,K,L,L-K-a) q^Kq^{K+1}\cdots q^{K+a-1}c(q,N,K+a+1,L-K-a)\\
+&\sum_{L=K+a+1}^N J(q,K,L,L-K-a) q^{K}q^{K+1}\cdots q^{K+a} c(q,N,K+a+1,L-K-a-1).
\end{align*}
Since $c(q,N,K+a+1,N-K-a)=0$, the two summations can be re--indexed into a single summation as
\begin{multline*}
\sum_{L=K+a}^N c(q,N,K+a+1,L-K-a) q^{K}q^{K+1}\cdots q^{K+a} \\
\times \left( J(q,K,L+1,L+1-K-a) + q^{-(K+a)}J(q,K,L,L-K-a)\right).
\end{multline*}
By \eqref{Recur}, this equals
$$
\sum_{L=K+a}^N c(q,N,K+a+1,L-K-a) q^{K}q^{K+1}\cdots q^{K+a} J(q,K,L+1,L-K-a),
$$
which also equals
$$
\sum_{L=K+a+1}^N c(q,N,K+a+1,L-K-a-1) \tilde{c}(q,N,K,L-K,a+1).
$$
Thus \eqref{Dis} is true, finishing the proof.
\end{proof}

Recall that $N^{(j)}(\eta)$ counts the number of species $1,\ldots,n+1-j$ particles that are weakly to the right of $0$. 

\begin{corollary}
Suppose that $\eta(t)$ evolves as the multi--species $q$--TAZRP, and begins in the initial conditions with infinitely many $i$th class particles at $-\mathcal{M}_i$ for some integers $0 < \mathcal{M}_n < \ldots < \mathcal{M}_1$. Fix non--negative integers $k_1,\ldots,k_n$.  Set $N=k_1+\ldots+k_n$ and define ${M}_j$ ($1 \leq j \leq N$) by 
$$
{M}_{k_1+\ldots+k_m+1} = {M}_{k_1+\ldots+k_m+2} = \cdots = {M}_{k_1+\ldots + k_{m+1}} = \mathcal{M}_{m+1}
$$
Then
$$
\mathbb{E}\left[ \prod_{j=1}^n q^{k_j N_0^{(n+1-j)}(\eta(t))}\right] = \frac{(-1)^N q^{N(N-1)/2}}{(2\pi i)^N}\int \frac{dw_1}{w_1} \cdots \int \frac{dw_N}{w_N} B(w_1,\ldots,w_N) \prod_{j=1}^N (1-w_j)^{-{M}_j+0} e^{-w_jt  }   .
$$
where the $w_r$--contour contains $qw_{r+1},\ldots, qw_n$ and $1$, but not $0$.
\end{corollary}
\begin{proof}
Use Proposition \ref{Dual} and part (b) of Proposition \ref{Main}. Take $K=1$ in the small--contour integral $\tilde{I}$. The term $\inv(\sigma)$ is simply $0$ since the only $\sigma$ is the identity. The resulting sum over $x_1$ is a geometric series, resulting in $z_1^{-1}(1-(1-z_1)^{-M_1})$. Because the integrand does not contain the pole $z_1=0$, the result follows. 
\end{proof}

\begin{remark}
By comparing the formula to Theorem 2.11 of \cite{BCSDuality}, we see the equality
$$
\mathbb{E}\left[ \prod_{j=1}^n q^{k_j N_0^{(n+1-j)}(\eta(t))}\right]  = \mathbb{E}_{\text{step}}\left[ \prod_{j=1}^n q^{x_{M_j}(t)+M_j}\right]
$$
where $x_M(t)$ denotes the position of the $M$--th particle in the single--species $q$--TASEP (which is equivalent to the single--species $q$--TAZRP). This proves the corollary.
\end{remark}

\section{Computer Simulations}
\subsection{The embedded Markov chain}

\begin{example}
Let $X=(x_1,x_2,x_3)=(0,-1,-2)$ and $Y=(y_1,y_2,y_3)=(1,3,2)$. Then
$$
\mathbb{P}_{X,\mathcal{H}}(Y)=\frac{q^{2}\left(17 q^{6}+349 q^{5}+2500 q^{4}+8610 q^{3}+15932 q^{2}+16454 q+7168\right)}{729\left(q^{8}+13 q^{7}+73 q^{6}+232 q^{5}+460 q^{4}+592 q^{3}+496 q^{2}+256 q+64\right)}
$$
For $X=(0,-2,-3)$ and $Y=(1,2,1)$, the value of $\mathbb{P}_{X,\mathcal{H}}(Y)$ is the same.
For $X=(0,-2,-2)$ and $Y=(1,2,2)$, the values of  $\mathbb{P}_{X,\mathcal{H}}(Y)$ is
$$\frac{q^{2}\left(89 q^{5}+903 q^{4}+3325 q^{3}+5905 q^{2}+5091 q+1697\right)}{729\left(q^{7}+11 q^{6}+51 q^{5}+130 q^{4}+200 q^{3}+192 q^{2}+112 q+32\right)}$$
Same for $X=(0,-1,-1)$ and $Y=(1,3,3)$ it is the same.
\end{example}

\begin{example}
Here is a non--example. Suppose $X=(0,-1,-2)$ and $Y=(1,3,4)$. Then 
$$
\mathbb{P}_{\mathcal{H},X}(Y)=\frac{q^{3}\left(808 q^{6}+9507 q^{5}+45927 q^{4}+119125 q^{3}+179061 q^{2}+151389 q+55513\right)}{6561\left(q^{9}+15 q^{8}+99 q^{7}+378 q^{6}+924 q^{5}+1512 q^{4}+1680 q^{3}+1248 q^{2}+576 q+128\right)}.
$$
However, for $X=(0,-1,-3)$ and $Y=(1,3,3)$
$$
\mathbb{P}_{\mathcal{H},X}(Y)=\frac{q^{2}\left(808 q^{6}+9507 q^{5}+45927 q^{4}+119125 q^{3}+179061 q^{2}+151389 q+55513\right)}{19683\left(q^{8}+13 q^{7}+73 q^{6}+232 q^{5}+460 q^{4}+592 q^{3}+496 q^{2}+256 q+64\right)}
$$
even though the intersection numbers are the same. Similarly, for $X=(0,-2,-2)$ and $Y=(1,2,4)$,
$$
\mathbb{P}_{\mathcal{H}}(X,Y)=\frac{q^{3}\left(162 q^{5}+1600 q^{4}+6011 q^{3}+11259 q^{2}+10793 q+4195\right)}{729\left(q^{8}+13 q^{7}+73 q^{6}+232 q^{5}+460 q^{4}+592 q^{3}+496 q^{2}+256 q+64\right)}
$$
\end{example}

\subsection{The continuous--time multi--species $q$--TAZRP}
\begin{example}
A few examples of the shift--invariance are shown below. Set $q=0.6$ and $t=2$. In the table below, the coordinates on the vertical axis show the initial conditions $x$, while the coordinates on the horizontal axis show $y-x$, where the probability is on the event that the bounds on $X(2)$ being less than or equal to $y$. The entries show the probability of $\mathbb{P}_x(X(2) \leq y)$, estimated from 10000000 samples. Some of the entries are blank because the author used up the allotted time on the Texas A\&M High Performance Research Computing Center.








\begin{center}
\begin{tabular}{| c|c|c|c|c|c|c|c|}
& (0,1,0) & (0,0,1) &  (0,0,2) & (0,1,4)& (0,1,3) & (0,1,2) & (0,1,1)  \\
\hline
(0,0,0) &  0.0332632 & 0.0279420  &0.0343266  &  0.0727376& 0.0695906 & 0.0626983 & 0.0513806 \\
\hline 
$(0,0,-1)$&  0.0100777 & 0.0278481  & 0.0343091& 0.0727394 &0.0695814  &0.0629251&0.0482535 \\
\hline 
$(0,-1,0)$ & 0.0278073 &  0.0100938 && 0.0727483 & 0.0673652 & 0.0582948 & \\
\hline
$(0,-1-1)$ & 0.0165454 & 0.0121191 & 0.0152055 & 0.0726808 & 0.0695527 & 0.0628361 & 0.0515214\\
\hline
$(0,0,-2)$ &&&& 0.0726787& 0.0695210 &  &\\
\hline
\end{tabular}
\end{center}
\end{example}

\begin{example}




The contour integral can also be evaluated explicit using residues.  For example, under the fifth column, when $y-x=(0,1,3)$, evaluating the residue explicitly yields,

$$\frac{1}{2(-1+q)^{4}(1+q)^{3}} e^{-\left(\left(6+2 q+q^{2}\right) t\right)}\left(2 e^{(5+q) t}+2 e^{\left(3+2 q+q^{2}\right) t} q^{3}(1+q)^{3}(-3+q-t+q t)-\right.$$
+\(\left.e^{\left(4+q+q^{2}\right) t} q\left(2+4 q^{2}(-3+t)-2 q^{4}(3+2 t)+q^{5}\left(2+t^{2}\right)-2 q^{3}\left(11+3 t+t^{2}\right)+q\left(6+6 t+t^{2}\right)\right)\right)\), which at $t=2$ and $q=0.6$  numerically evaluates to 0.0695753.... This is approximately the values in the fifth column.

Similarly, for the fourth column, the residues yield
\begin{equation}
\begin{array}{l}\frac{1}{6(-1+q)^{5}(1+q)^{4}} \\ e^{-\left(\left(6+2 q+q^{2}\right) t\right)}\left(-6 e^{(5+q) t}+3 e^{\left(3+2 q+q^{2}\right) t} q^{3}(1+q)^{4}\left(12+6 t+t^{2}-2 q(2+t)^{2}+q^{2}\left(2+2 t+t^{2}\right)\right)-\right. \\ e^{\left(4+q+q^{2}\right) t} q\left(-6+q^{2}(60-18 t)+6 q^{6} t+6 q^{4}(15+2 t)-3 q^{5} t\left(2+4 t+t^{2}\right)+\right. \\ \left.\left.q^{7}\left(6+6 t+3 t^{2}+t^{3}\right)+3 q^{3}\left(52+8 t+5 t^{2}+t^{3}\right)-q\left(24+24 t+6 t^{2}+t^{3}\right)\right)\right)\end{array}
\end{equation}
which evaluates to 0.0727076....
\end{example}





















\bibliographystyle{alpha}
\bibliography{MultispeciesqBosonDraft}

\newcommand{\etalchar}[1]{$^{#1}$}
\def\cprime{$'$} \def\Dbar{\leavevmode\lower.6ex\hbox to 0pt{\hskip-.23ex
  \accent"16\hss}D} \def\cprime{$'$}
\begin{thebibliography}{CdGH{\etalchar{+}}}

\bibitem[ACQ11]{ACQ}
Gideon Amir, Ivan Corwin, and Jeremy Quastel.
\newblock Probability distribution of the free energy of the continuum directed
  random polymer in 1 + 1 dimensions.
\newblock {\em Communications on Pure and Applied Mathematics}, 64(4):466--537,
  2011.

\bibitem[AHR98]{AHR98}
Peter~F Arndt, Thomas Heinzel, and Vladimir Rittenberg.
\newblock Spontaneous breaking of translational invariance in one-dimensional
  stationary states on a ring.
\newblock {\em Journal of Physics A: Mathematical and General}, 31(2):L45--L51,
  jan 1998.

\bibitem[Bar15]{BARRAQUAND20152674}
Guillaume Barraquand.
\newblock A phase transition for q-tasep with a few slower particles.
\newblock {\em Stochastic Processes and their Applications}, 125(7):2674--2699,
  2015.

\bibitem[BCS14]{BCSDuality}
Alexei Borodin, Ivan Corwin, and Tomohiro Sasamoto.
\newblock From duality to determinants for $q$-{TASEP} and {ASEP}.
\newblock {\em Ann. Probab.}, 42(6):2314--2382, 11 2014.

\bibitem[BDJ99]{bdj1}
Jinho Baik, Percy Deift, and Kurt Johansson.
\newblock On the distribution of the length of the longest increasing
  subsequence of random permutations.
\newblock {\em J. Amer. Math. Soc.}, 12(4):1119--1178, 1999.

\bibitem[BGW]{BGW}
Alexei Borodin, Vadim Gorin, and Michael Wheeler.
\newblock Shift-invariance for vertex models and polymers.
\newblock {\em Proceedings of the London Mathematical Society}.

\bibitem[CdGH{\etalchar{+}}]{AHR}
Zeying Chen, Jan de~Gier, Iori Hiki, Tomohiro Sasamoto, and Masato Usui.
\newblock Limiting current distribution for a two species asymmetric exclusion
  process.

\bibitem[dGMW]{GMW}
Jan de~Gier, William Mead, and Michael Wheeler.
\newblock Transition probability and total crossing events in the multi-species
  asymmetric exclusion process.

\bibitem[Dim]{Dim}
Evgeni Dimitrov.
\newblock Two-point convergence of the stochastic six-vertex model to the airy
  process.

\bibitem[DM97]{Der97}
B~Derrida and K~Mallick.
\newblock Exact diffusion constant for the one-dimensional partially asymmetric
  exclusion model.
\newblock {\em Journal of Physics A: Mathematical and General},
  30(4):1031--1046, feb 1997.

\bibitem[FV15]{FVAIHP}
Patrik~L. Ferrari and Balint Veto.
\newblock {T}racy--{W}idom asymptotics for $q$-{TASEP}.
\newblock {\em Ann. Inst. H. Poincaré Probab. Statist.}, 51(4):1465--1485, 11
  2015.

\bibitem[Gal]{Gal}
Pavel Galashin.
\newblock Symmetries of stochastic colored vertex models.

\bibitem[GO]{PatSurvey}
Patricia Goncalves and Alessandra Occelli.
\newblock On energy solutions to stochastic burgers equation.

\bibitem[Hay]{Haya}
Kohei Hayaski.
\newblock Equilibrium fluctuations for totally asymmetric interacting particle
  systems.

\bibitem[Joh00]{JohShape}
Kurt Johansson.
\newblock Shape {F}luctuations and {R}andom {M}atrices.
\newblock {\em Comm. Math. Phys}, 209:437--476, 2000.

\bibitem[KL14]{KoLe}
Marko Korhonen and Eunghyun Lee.
\newblock The transition probability and the probability for the left-most
  particle's position of the q-totally asymmetric zero range process.
\newblock {\em Journal of Mathematical Physics}, 55(1):013301, 2014.

\bibitem[KPZ86]{KPZ}
Mehran Kardar, Giorgio Parisi, and Yi-Cheng Zhang.
\newblock Dynamic scaling of growing interfaces.
\newblock {\em Physical Review Letters}, 56(9):889--892, 1986.

\bibitem[KS60]{KS60}
John~G. Kemeny and J.~Laurie Snell.
\newblock {\em Finite Markov Chains}.
\newblock Springer, 1960.

\bibitem[Kua17]{KIMRN}
Jeffrey Kuan.
\newblock A multi-species {ASEP}$(q,j)$ and $q$-{TAZRP} with stochastic
  duality.
\newblock {\em International Mathematics Research Notices},
  2018(17):5378--5416, 2017.

\bibitem[Kua18]{KuanCMP}
Jeffrey Kuan.
\newblock An algebraic construction of duality functions for the stochastic
  ${U}_q({A}_n^{(1)})$ vertex model and its degenerations.
\newblock {\em Communications in Mathematical Physics}, 359(1):121--187, Apr
  2018.

\bibitem[Kua19]{KuanAHP}
Jeffrey Kuan.
\newblock Probability distributions of multi-species q-{TAZRP} and {ASEP} as
  double cosets of parabolic subgroups.
\newblock {\em Annales Henri Poincar\'{e}}, 20(4):1149--1173, April 2019.

\bibitem[Lig76]{Ligg76}
Thomas~M. Liggett.
\newblock Coupling the simple exclusion process.
\newblock {\em Annals of Probability}, 4(3):339--356, 1976.

\bibitem[QS]{QS}
Jeremy Quastel and Sourav Sarkar.
\newblock Convergence of exclusion processes and kpz equation to the kpz fixed
  point.

\bibitem[Spo14]{Spohn}
Herbert Spohn.
\newblock Nonlinear fluctuating hydrodynamics for anharmonic chains.
\newblock {\em Journal of Statistical Physics}, 154:1191–1227, 2014.

\bibitem[SW98]{SW98}
Tomohiro Sasamoto and Miki Wadati.
\newblock Exact results for one--dimensional totally asymmetric diffusion
  models.
\newblock {\em Journal of Physics A}, 31(28):6057--6071, 1998.

\bibitem[Tak]{Take15}
Yoshihiro Takeyama.
\newblock Algebraic construction of multi--species $q$--{B}oson system.

\bibitem[TW94]{TW94}
Craig~A. Tracy and Harold Widom.
\newblock Level-spacing distributions and the {A}iry kernel.
\newblock {\em Communications in Mathematical Physics}, 159(1):151--174, Jan
  1994.

\bibitem[WW09]{WW}
Jon Warren and Peter Windridge.
\newblock Some examples of dynamics for {G}elfand-{T}setlin patterns.
\newblock {\em Electron. J. Probab.}, 14:no. 59, 1745--1769, 2009.

\bibitem[WW16]{Wawa}
Dong Wang and David Waugh.
\newblock The transition probability of the q-{TAZRP} (q-bosons) with
  inhomogeneous jump rates.
\newblock {\em SIGMA}, 12(037):16, 2016.

\end{thebibliography}

\end{document}